\documentclass[12pt,reqno]{amsart} 
\usepackage{graphicx}
\usepackage{amssymb,amsmath}
\usepackage{amsthm}
\usepackage{color}
\usepackage{hyperref}
\usepackage{epstopdf}

\usepackage{subfigure} 
\usepackage{caption}

\usepackage{verbatim} 
	
\usepackage{relsize} 
	
	\newcommand{\be}{\begin{equation}}
		
		\newcommand{\ee}{\end{equation}}

	\newcommand{\R}{{\mathbb R}}

	\newcommand{\Ker}{{\rm \,Ker}}

	\newcommand{\bigslant}[2]{{\raisebox{.1em}{$#1$}\left/\raisebox{-.1em}{$#2$}\right.}}

	\numberwithin{equation}{section}
	\numberwithin{figure}{section}
	
	\newtheorem{theorem}{Theorem}[section]
	\newtheorem{proposition}[theorem]{Proposition}
	\newtheorem{remark}[theorem]{Remark}
	\newtheorem{lemma}[theorem]{Lemma}
	\newtheorem{corollary}[theorem]{Corollary}
		\newtheorem{conjecture}[theorem]{Conjecture}

\title[Stability of periodic waves in the NLS-IDD model]{Stability of periodic waves in the model \\ with intensity--dependent dispersion}

\author[F. Natali]{F\'abio Natali}
\address[F. Natali]{Department of Mathematics - State University of Maring\'a, 
Maring\'a, Paran\'a, Brazil, 87020900}
\email{fmanatali@uem.br}

\author[D. E. Pelinovsky]{Dmitry E. Pelinovsky}
\address[D. E. Pelinovsky]{Department of Mathematics and Statistics, McMaster University, Hamilton, Ontario, Canada, L8S 4K1}
\email{pelinod@mcmaster.ca}
		
\author[S. Wang]{Shuoyang Wang}
\address[S. Wang]{Department of Mathematics and Statistics, McMaster University, Hamilton, Ontario, Canada, L8S 4K1}
\email{wangs455@mcmaster.ca}

\begin{document}
	
	\maketitle
	

\begin{abstract} 
	We study standing periodic waves modeled by the nonlinear Schr\"{o}dinger equation with the intensity-dependent dispersion coefficient. 
	Spatial periodic profiles are smooth if the frequency of the standing waves 
	is below the limiting frequency, for which the profiles become peaked (piecewise continuously differentiable with a finite jump of the first derivative). We prove that there exist two families of the periodic waves with smooth profiles separated by a homoclinic orbit and the period function (the energy-to-period mapping) is monotonically increasing for the family inside the homoclinic orbit and decreasing for the family outside the homoclinic orbit. This property allows us to derive a sharp criterion for the energetic stability of such standing periodic waves under time evolution if the perturbations are periodic with the same period for both families and, additionally, for the family outside the homoclinic orbit, spatially odd with respect to the half-period. By numerically approximating the sharp stability criterion, we show that both families are energetically stable for small frequencies but become unstable when the frequency approaches the limiting frequency of the peaked waves. 
\end{abstract}

\section{Introduction}

We consider the nonlinear Schr\"odinger (NLS) equation, where the dispersion coefficient depends linearly on the wave intensity. This model in one spatial dimension 
can be written in the normalized form:
\begin{equation}
	iu_t+(1-|u|^2)u_{xx}+|u|^2u=0, 
	\label{NLS-IDD}
	\end{equation}
where $u = u(t,x)$ and $u:\mathbb{R}\times \mathbb{R}\rightarrow \mathbb{C}$. 
We assume that $u(t,\cdot)$ is spatially periodic with the period $L$ for any $t \in \mathbb{R}$. If the dispersion coefficient is constant, the model is equivalent to the cubic 
focusing NLS equation, one of the fundamental models of nonlinear science \cite{Fibich,Kev-Dark-2015}. We refer to (\ref{NLS-IDD}) as {\em the NLS--IDD equation}.

\subsection{Background and motivations}

Mathematical models with the intensity-dependent dispersion terms have been studied in the physics of the coherently prepared multistate atoms \cite{greentree}, quantum well waveguides \cite{koser}, fiber-optics communication systems \cite{OL2020}, and the quantum harmonic oscillators in the presence of
nonlinear effective masses~\cite{chang}. 

The NLS--IDD equation also arises as the continuum limit of the Salerno lattice model ~\cite{salerno}, 
\begin{equation}
i \partial_{\tau} \psi_n + (1-|\psi_n|^2) (\psi_{n+1} + \psi_{n-1}) + \mu |\psi_n|^2 \psi_n = 0,
\label{Salerno}
\end{equation}
where $\mu \in \R$ is the coefficient of the onsite nonlinearity 
and $\psi_n = \psi_n(\tau)$ is the wave function in $(\tau,n) \in \R \times \mathbb{Z}$. 
If $\mu = 2 + h^2$ and $\psi_n(\tau) = e^{2i \tau} u(h^2 \tau, hn)$ with a smooth $u = u(t,x)$, then expanding in powers of the small stepsize $h$ yields the NLS--IDD equation (\ref{NLS-IDD}) from the Salerno model (\ref{Salerno}) at order $\mathcal{O}(h^2)$.

The mathematical analysis of model (\ref{NLS-IDD}) 
without the local cubic term $|u|^2 u$ was developed in 
\cite{RKP}, where it was shown that a continuous family of bright solitons exists among the standing wave solutions. The spatial profiles of bright solitons have two logarithmic singularities for the first derivative and the continuous parameter is given by the distance between the two singularities. The energetic stability of the bright solitons was obtained in \cite{PRK} by using the variational characterization of the singular profiles as minimizers of the mass subject to a fixed energy. Well-posedness of the model was not studied in \cite{PRK,RKP}. 

A similar model without the local cubic term $|u|^2 u$ and with the inverted 
intensity--dependent coefficient $(1-|u|^2)^{-1} u_{xx}$ was considered in \cite{PP24}, where a family of dark solitons (traveling wave solutions) was shown to have smooth spatial profiles and the limiting black solitons (standing wave solutions) were shown to be energetically stable as constrained minimizers 
of the energy subject to fixed mass and momentum. Dark solitons in the quasilinear NLS equations with nonconstant dispersion terms were considered in \cite{Laire1,Laire2,Laire3}. Both bright and dark solitons were also studied in the NLS equations with regularized dispersion terms \cite{Albert,Sparber,Plum}.

The NLS--IDD equation (\ref{NLS-IDD}) was studied in \cite{KPR24}, where the continuous family of bright solitons is parameterized by the frequency of the standing wave solution $u(t,x) = e^{i \omega t} \phi(x)$ with the spatially decaying pofile $\phi$. The profile $\phi$ smooth for $0 < \omega < 1$ and peaked (piecewise continuously differentiable with a single jump of the first derivative) for $\omega = 1$. A sharp criterion for energetic stability of bright solitons with respect to the spatially decaying perturbations in $H^1(\R)$ was obtained in \cite{KPR24} from 
the variational characterization of the smooth profiles as local minimizers of the energy subject to a fixed mass. The sharp criterion is given by the monotone increase of the mass with respect to the frequency, the latter condition is checked numerically. 

Energetic stability is equivalent to the orbital stability if the local 
well-posedness of the NLS-IDD equation (\ref{NLS-IDD}) can be obtained in $H^1(\R)$. However, the state-of-the-art in the well-posedness of quasilinear NLS equations is not yet at the level of $H^1(\R)$. Local well-posedness of the models which include (\ref{NLS-IDD}) was proven in Sobolev spaces of higher regularity \cite{KPV03,Marzuola1,Popen}. More recently, the local well-posedness of quasilinear NLS equations was established in $H^s(\R)$ for $s > 2$ in \cite{Marzuola} and for small data in $H^s(\R)$ for $s > 1$ in \cite{Ifrim}. 
Local well-posedness of quasilinear NLS equations including the NLS--IDD equation (\ref{NLS-IDD}) was also extended to the periodic domain in Sobolev spaces of higher regularity \cite{Feola1,Feola2,Feola}.

{\em The main purpose of this work is to study the energetic stability 
of standing periodic waves with the smooth profiles with respect to periodic perturbations of the same period. }
The periodic spatial domain is more practical for physical experiments modeled by the NLS--IDD equation (\ref{NLS-IDD}). 
The mathematical analysis of stability in the periodic setting introduces additional challenges because the Morse index in the variational characterization may exceed a single negative eigenvalue. We control the Morse index with a precise analysis of the monotonicity of the period function (the energy-to-period mapping). 
Similarly to the scopes of \cite{KPR24}, we obtain a sharp criterion for the energetic stability of the smooth periodic waves as local minimizers of the 
energy in $H^1_{\rm per}$ subject to fixed mass, provided that the mass 
at the periodic wave profile is monotonically increasing with respect to 
frequency $\omega$ for $0 < \omega < 1$. We compute the latter criterion numerically and point out inaccuracies in the previous numerical approximations in \cite{KPR24} performed for the case of bright solitons. These main results of our study are described next.

\subsection{Main results}

We denote the space of square integrable $L$-periodic functions by $L^2_{\rm per}$. For $s\geq0$, the Sobolev space $H^s_{\rm per}$ is the set of periodic distributions such that
$$
\|f\|_{H^s_{\rm per}} :=  \left( \sum_{k=-\infty}^{\infty}(1+|k|^2)^s|\hat{f}(k)|^2 \right)^{1/2} <\infty,
$$
where $\hat{f}$ is the periodic Fourier transform of $f$ (the Fourier series of $f$). The space $H^s_{\rm per}$ is a  Hilbert space with a natural inner product denoted by $\langle \cdot, \cdot \rangle_{H^s_{\rm per}}$. When $s=0$, the space $H^s_{\rm per}$ is isometrically isomorphic to the space  $L^2_{\rm per}$, that is, $L^2_{\rm per}=H^0_{\rm per}$. The norm and inner product in $L^2_{\rm per}$ are denoted by $\|\cdot \|_{L^2_{\rm per}}$ and $\langle \cdot, \cdot \rangle_{L^2_{\rm per}}$.

The time-dependent NLS--IDD equation (\ref{NLS-IDD}) admits the conserved energy $H(u)$ and mass $Q(u)$ given by 
\begin{equation}
H(u) = \int_{\mathbb{T}_L} \left( |u_x|^2+|u|^2+\log(1-|u|^2) \right) dx
\label{Eu}
\end{equation}
and
\begin{equation}
Q(u) = - \int_{\mathbb{T}_L} \log(1-|u|^2)dx, 
\label{Fu}
\end{equation}
where $\int_{\mathbb{T}_L}$ denotes the integral over the periodic domain $\mathbb{T}_L$ with the spatial period $L$, which is independent on the starting point of integration. The conserved quantities are well defined in the set of functions
$$
\mathcal{X} = \left\{ u \in H^1_{\rm per} : \quad \| u \|_{L^{\infty}} < 1 \right\}.
$$
The NLS--IDD equation (\ref{NLS-IDD}) also admits the conserved momentum $P(u)$ if $u \neq 0$, see \cite{KPR24}. Since the momentum does not play any role in our study, we do not introduce it here. 

We consider standing waves of the form $u(t,x)=e^{i\omega t}\phi(x)$, where $\omega$ is the wave frequency. Substituting this ansatz into $(\ref{NLS-IDD})$, we obtain
\begin{equation}
	-(1-\phi^2)\phi''+\omega\phi-\phi^3=0,
	\label{ode}
\end{equation}
which can be rewritten as Newton's equation for a 1D particle in a potential energy $V$:
\begin{equation}
\label{odephi}
\frac{d^2 \phi}{dx^2} = \frac{(\omega - \phi^2)}{1-\phi^2} \phi = -\frac{dV}{d\phi}, \qquad 
V(\phi) = \frac{1}{2} (\omega - \phi^2) + \frac{1}{2} (1-\omega) \log \frac{1 - \omega}{1 - \phi^2}.
\end{equation}
The total energy $E$ of Newton's particle is conserved along every solution 
of (\ref{odephi}):
\begin{equation}
\label{energy}
E(\phi,\phi') = \frac{1}{2} (\phi')^2 + V(\phi).
\end{equation}
The variational characterization of the spatial profile $\phi$ is possible since 
the second-order equation (\ref{odephi}) is the Euler--Lagrange equation 
for the augmented energy functional
\begin{equation}
	G(u) = H(u) + \omega Q(u),
	\label{Gu}
\end{equation}
defined from the conserved energy $H(u)$ and mass $Q(u)$ in (\ref{Eu}) and (\ref{Fu}).
	
The phase portrait in Figure \ref{fig-phase} represents all bounded solutions of 
the system (\ref{odephi}) for $0 < \omega < 1$, see also \cite{KPR24}. There exist two 
families of periodic orbits with smooth profiles separated by a pair of homoclinic orbits. One family is inside one of the two homoclinic orbits with the left (negative) periodic orbits being symmetrically reflected from the right (positive) periodic orbits due to 
the symmetry transformation: $\phi \to -\phi$. The other family is outside the two homoclinic orbits and symmetrically span all four quadrants of the phase plane.

\begin{figure}[htb!]
	\centering
	\includegraphics[width=0.6\textwidth,height=0.3\textheight]{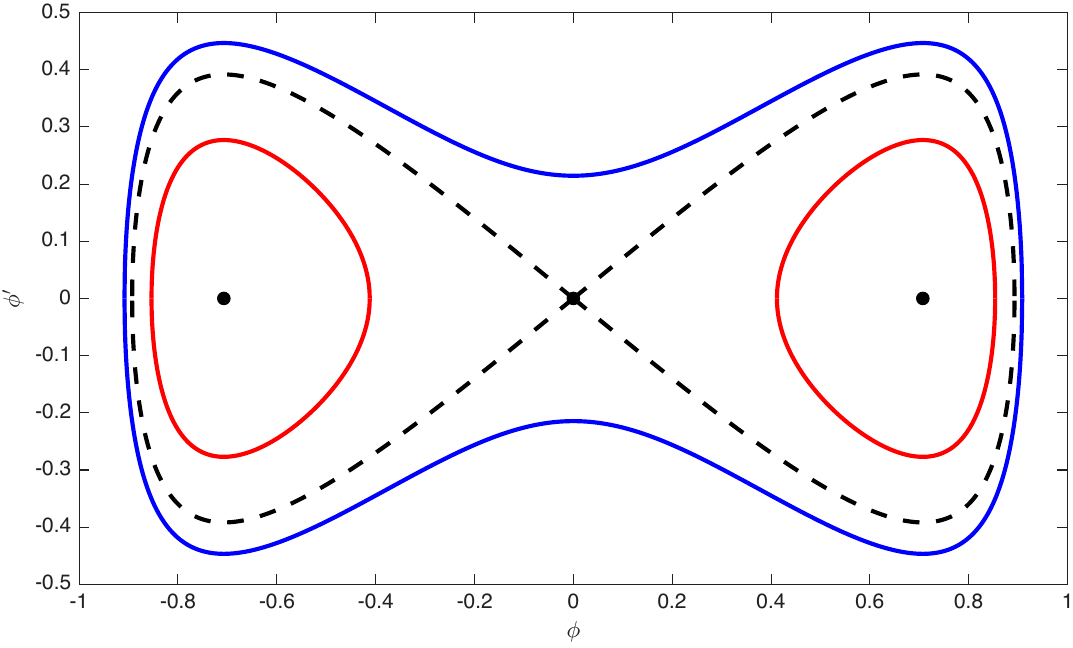} 
	\caption{The phase portrait of system (\ref{odephi}) for $\omega = 0.5$.}
	\label{fig-phase}
\end{figure}

The following theorem summarizes the existence properties of the two families of the periodic orbits. 

\begin{theorem}
	\label{th-existence}
	Fix the spatial period $L > 0$ for the periodic domain $\mathbb{T}_L$ and define 
	$$
	\omega_L = \frac{2 \pi^2}{L^2 + 2 \pi^2}, \quad \Omega_L = -\frac{4 \pi^2}{L^2}.
	$$
	For any $\omega \in (\omega_L,1)$, 
	there exists a periodic orbit of system (\ref{odephi}) with the smooth profile $\phi$ satisfying
	\begin{equation}
	\label{even-wave}
	\left\{ \begin{array}{lll} 0 < \phi(x) < 1, \quad & & \forall\ x \in \mathbb{T}_L,\\ \phi(x-x_0) = \phi(x_0-x), \quad & x_0 \in \mathbb{T}_L, \quad & \forall\ x \in \mathbb{T}_L.
	\end{array} \right.
	\end{equation}
	For any $\omega \in (\Omega_L,1)$, 
	there exists a periodic orbit of system (\ref{odephi}) with the smooth profile $\phi$ satisfying 
	\begin{equation}
	\label{odd-wave}
		\left\{ \begin{array}{lll} 
	-1 < \phi(x) < 1, \quad & & \forall\ x \in \mathbb{T}_L, \\
\phi(x-x_0) = -\phi(x_0-x) = \phi\left(\frac{L}{2} - x+x_0 \right), \quad & 
x_0 \in \mathbb{T}_L, \quad & \forall\ x \in \mathbb{T}_L.
	 	\end{array} \right.
	\end{equation}
	For both families, $x_0$ is an arbitrary translational parameter along the periodic orbit. 
\end{theorem} 

\begin{remark}
For simplicity of terminology, we call the family of periodic orbits inside the homoclinic orbits satisfying (\ref{even-wave}) as the even waves and the family of periodic orbits outside the homoclinic orbits satisfying (\ref{odd-wave}) as the odd waves. Figure \ref{fig-phase} shows a former periodic orbit in blue and a latter periodic orbit in red together with its symmetric reflection. The homoclinic orbits are shown by dashed black lines.
\end{remark}
	
Each family of periodic orbits correspond to the energy level $E(\phi,\phi') = \mathcal{E}$ given by the first invariant (\ref{energy}). For $\omega \in (0,1)$, the family of even waves satisfying 
(\ref{even-wave}) corresponds to $\mathcal{E} \in (0,\mathcal{E}_{\omega})$ and 
the family of odd waves satisfying (\ref{odd-wave}) corresponds to $\mathcal{E} \in (\mathcal{E}_{\omega},\infty)$, where 
$$
\mathcal{E}_{\omega} = V(0) = \frac{1}{2} \omega + \frac{1}{2} (1-\omega) \log(1-\omega)
$$
is the energy level corresponding to the homoclinic orbits for the saddle point $(0,0)$. If $\omega \in (-\infty,0)$, the family of odd waves satisfying (\ref{odd-wave}) correspond to $\mathcal{E}\in (\mathcal{E}_{\omega},\infty)$, 
where $\mathcal{E}_{\omega} = V(0)$ is the energy level corresponding to the center point $(0,0)$. For each energy level $E(\phi,\phi') = \mathcal{E}$, 
we can define the period function $T(\mathcal{E},\omega)$ by 
\begin{equation}
\label{period-function}
T(\mathcal{E},\omega) = \oint \frac{d\phi}{\sqrt{2 (\mathcal{E} - V(\phi))}}, 
\end{equation}
where $\oint$ corresponds to the line integral taken along the closed periodic orbit. 
Figure \ref{fig-period} shows the dependence of $T(\mathcal{E},\omega)$ versus 
$\mathcal{E}$ for fixed values of $\omega \in (0,1)$, where the divergence of $T(\mathcal{E},\omega)$ corresponds 
to the homoclinic orbit at $\mathcal{E} = \mathcal{E}_{\omega}$. The figure suggests that, for $\omega \in (0,1)$, the mapping $\mathcal{E}\to T(\mathcal{E},\omega)$ is monotonically increasing for the even wave and is monotonically decreasing for the odd wave. These properties 
are formulated in the following theorem.

\begin{figure}[htb!]
	\centering
	\includegraphics[width=0.8\textwidth,height=0.4\textheight]{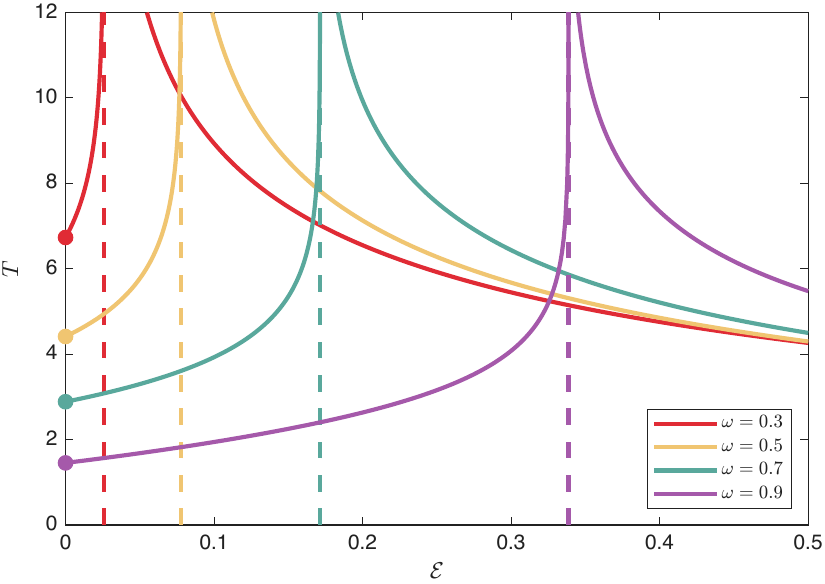} 
	\caption{The period function $T(\mathcal E,\omega )$ versus $\mathcal{E}$ for fixed values of $\omega$. The dots denote the cutoff value of $\mathcal{E}$ satisfying $T(\mathcal{E},\omega) = \pi \sqrt{2(1-\omega)/\omega}$ for $\omega = \omega_L$. The vertical lines show divergence of $T(\mathcal{E},\omega)$ at 
	$\mathcal{E} = \mathcal{E}_{\omega}$.}
		\label{fig-period}
\end{figure}

\begin{theorem}
	\label{th-period}
	The period function $T = T(\mathcal{E},\omega)$ in (\ref{period-function}) 
	is a $C^1$ function of $\mathcal{E} \in (0,\infty) \backslash \mathcal{E}_{\omega}$ if $\omega \in (0,1)$ 
	and $\mathcal{E} \in (\mathcal{E}_{\omega},\infty)$ if $\omega \in (-\infty,0)$. For any $\omega \in (0,1)$, the mapping $(0,\mathcal{E}_{\omega}) \ni \mathcal{E} \to T(\mathcal{E},\omega)$ is monotonically increasing. For any $\omega \in (-\infty,1)$, the mapping $(\mathcal{E}_{\omega},\infty) \ni \mathcal{E} \to T(\mathcal{E},\omega)$ is monotonically decreasing. 
\end{theorem} 

Due to smoothness and monotonicity of the period function in Theorem \ref{th-period}, one can uniquely define the energy level $\mathcal{E} = \mathcal{E}_L(\omega)$ for any spatial period $L > 0$ in Theorem \ref{th-existence} from the root of $T(\mathcal{E}_L(\omega),\omega) = L$, where $\omega \in (\omega_L,1)$ for the even wave and $\omega \in (\Omega_L,1)$ for the odd wave. Furthermore, the mappings $(\omega_L,1) \ni \omega \to \mathcal{E}_L(\omega)$ and $(\Omega_L,1) \ni \omega \to \mathcal{E}_L(\omega)$ are $C^1$. These smoothness properties play a central role in the energetic stability analysis of the periodic waves.

The Hessian operator $\mathcal{L} = H''(\phi) + \omega Q''(\phi)$ 
of the augmented energy functional (\ref{Gu}) computed at the critical point with the profile $\phi$ is defined as
	\begin{equation}
	\label{matrixop}
		\mathcal{L} = \left(\begin{array}{cccc} \mathcal{L}_{+}& 0\\\\
							0 & \mathcal{L}_{-}\end{array}\right), \quad 
\begin{array}{l} 
\mathcal{L}_+ = -\partial_x^2+1+(\omega-1)\frac{1+\phi^2}{(1-\phi^2)^2}, \\
\mathcal{L}_- = -\partial_x^2+1+(\omega-1)\frac{1}{(1-\phi^2)}, 
\end{array}
	\end{equation}
For simplicity of notations, we set
$$
\mathbb{H}^s_{\rm per}:= H^s_{\rm per} \times H^s_{\rm per}, \quad 
\mathbb{L}^2_{\rm per}:= L^2_{\rm per} \times L^2_{\rm per},
$$
endowed with their usual norms and scalar products. When necessary and since $\mathbb{C}$ can be identified with $\mathbb{R}^2$, notations above can also be used for complex-valued functions in the following sense: for $f\in \mathbb{H}_{\rm per}^s$ we have $f=f_1+if_2$ with $f_1, f_2 \in H_{\rm per}^s$. 

By studying the spectrum of $\mathcal{L}$ in $\mathbb{L}^2_{\rm per}$, we obtain the sharp criterion for the energetic stability of the periodic waves with the spatial profile $\phi$ stated in the following theorem. 

\begin{theorem}
	\label{th-stability}
	Fix the spatial period $L > 0$ as in Theorem \ref{th-existence} and set $x_0 = 0$. The profile $\phi \in H^1_{\rm per}$ is a $C^1$ function of $\omega$ for the even wave in $(\omega_L,1)$ and for the odd wave in $(\Omega_L,1)$. 
	For any $\omega \in (\omega_L,1)$, the even wave with the profile $\phi$ is a local minimizer of energy $H(u)$ for a fixed mass $Q(u)$ in $H^1_{\rm per}$, which is degenerate only due to translational and rotational symmetries, if and only if the mapping 
	$\omega \to Q(\phi)$ is monotonically increasing. 
	For any $\omega \in (\Omega_L,1)$, the odd wave with the profile $\phi$ is a local minimizer of energy $H(u)$ for a fixed mass $Q(u)$ in $\mathcal{Y}\subset H^1_{\rm per}$, where 
\begin{equation}
\label{Y-space}
	\mathcal{Y} = \left\{ u \in H^1_{\rm per} : \quad u\left(\frac{L}{2}-x\right) = -u\left(x - \frac{L}{2}\right), \quad \forall\ x \in \mathbb{T}_L \right\},
\end{equation}
 	which is only degenerate by the rotational symmetry, if and only if the mapping 
	$\omega \to Q(\phi)$ is monotonically increasing. 
\end{theorem}

By using accurate numerical approximations based on the first invariant (\ref{energy}) and the period function (\ref{period-function}), we can compute solutions of the implicit equation $T(\mathcal{E}_L(\omega),\omega) = L$ for a fixed spatial period $L > 0$ and the approximations of the spatial profile $\phi$ of the periodic wave. 

Figure \ref{fig-even} shows the corresponding results 
for the even wave satisfying (\ref{even-wave}) with $x_0 = 0$. The left panel plots $\tilde{\mathcal{E}}_L(\omega) := \mathcal{E}_L(\omega) - \mathcal{E}_{\omega}$ versus $\omega$ in $(\omega_L,1)$ for $L = 2\pi,3\pi,4\pi$ 
and the right panel shows the spatial profile $\phi = \phi(x)$ versus $x$ 
for $L = 4\pi$ and $\omega = 0.3, 0.6, 0.9$. Numerical inaccuracies in the computations occur near $\omega = 1$ and the end points in the numerical data on the left panel are shown by solid dots. The spatial profile of the even periodic wave becomes peaked as $\omega \to 1$. Solving (\ref{ode}) for $\omega = 1$ yields the peaked profile
\begin{equation}
\label{phi-limeven}
\omega = 1 : \quad \phi(x) =  \frac{\cosh \left (\frac{L}{2} - |x| \right)}{\cosh \left (\frac{L}{2} \right )} , \quad \quad x \in \left [-\frac{L}{2}, \frac{L}{2} \right],
\end{equation}
which is shown on the right panel by dashed line. 
The corresponding energy level can be computed as 
\begin{equation}\label{E-limeven}
\mathcal{E}_L(\omega = 1) = - \frac{1}{2\cosh^2 \left(\frac{L}{2}\right)},
\end{equation}
which is shown on the left panel by open dots. An interpolation between 
the right solid dot and the open dot for (\ref{E-limeven}) is shown by 
dotted line.

\begin{figure}[htb!]
	\centering
	\includegraphics[width=0.94\textwidth,height=0.35\textheight]{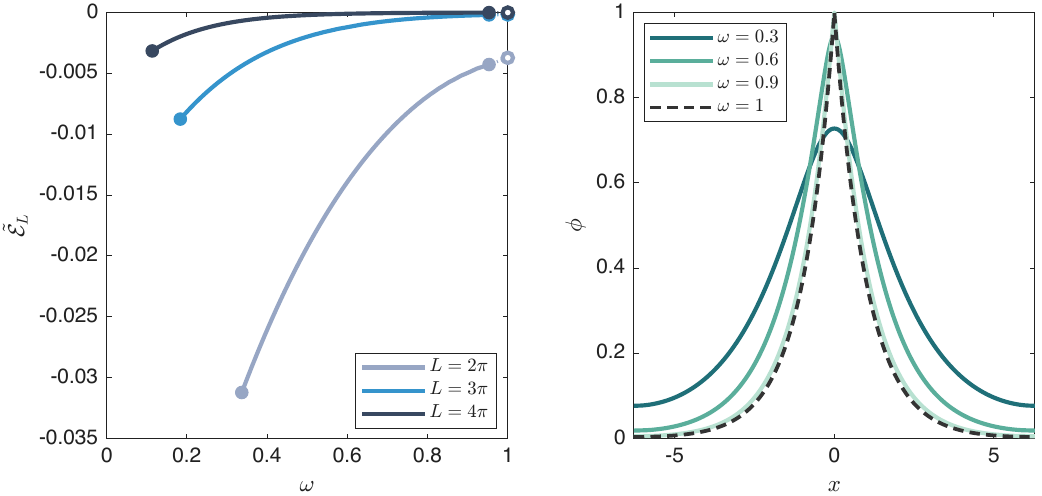} 
	\caption{Numerical approximations for the even waves satisfying (\ref{even-wave}) with $x_0 = 0$. Left: the dependence of 
		$\tilde{\mathcal{E}}_L$ versus $\omega$ for $L = 2\pi,3\pi,4\pi$. Right: the spatial profile $\phi$ versus $x$ for $\omega = 0.3,0.6,0.9$ and $L = 4\pi$. }
	\label{fig-even}
\end{figure}

Figure \ref{fig-odd} shows the corresponding results 
for the odd wave satisfying (\ref{odd-wave}) with $x_0 = 0$ for $\omega \in (\Omega_L,1)$ with $\Omega_L < 0$. We note the non-monotone dependence of $\tilde{\mathcal{E}}_L(\omega) := \mathcal{E}_L(\omega) - \mathcal{E}_{\omega}$ versus $\omega$ on the left panel, which is not an obstacle to our analysis.  
The spatial profile of the periodic wave becomes peaked as $\omega \to 1$. Solving (\ref{ode}) for $\omega = 1$ yields the odd spatial profile in the form:
\begin{equation}\label{phi-limodd}
\omega = 1 : \quad  \phi(x) = \left\{
\begin{array}{ll}
\displaystyle -\frac{\sinh \left ( \frac{L}{2} + x  \right )}{\sinh \left (\frac{L}{4} \right )} &   x \in \left [ -\frac{L}{2},  -\frac{L}{4} \right] \\[10pt]
\displaystyle \frac{\sinh x}{\sinh \left (\frac{L}{4} \right )} & x \in \left [-\frac{L}{4}, \frac{L}{4} \right] \\[12pt]
\displaystyle \frac{\sinh \left ( \frac{L}{2}-x \right )}{\sinh \left (\frac{L}{4} \right )}  &  x \in \left [ \frac{L}{4}, \frac{L}{2} \right]
\end{array}, 
\right.
\end{equation}
which is shown on the right panel by dashed line. 
The corresponding energy level can be computed as 
\begin{equation}\label{E-limodd}
\mathcal{E}_L(\omega = 1) = \frac{1}{2\sinh^2 \left(\frac{L}{2}\right)},
\end{equation}
which is shown on the left panel by open dots. The end points in the numerical data on the left panel are shown by solid dots.  An interpolation between 
the right solid dot and the open dot for (\ref{E-limodd}) is shown by 
dotted line.

\begin{figure}[htb!]
	\centering
	\includegraphics[width=0.94\textwidth,height=0.35\textheight]{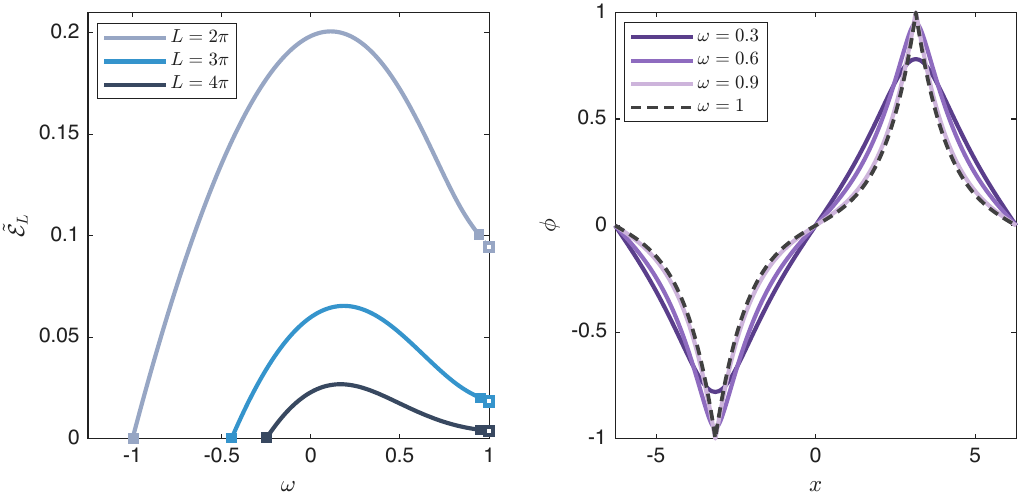} 
	\caption{Numerical approximations for the odd waves satisfying (\ref{odd-wave}) with $x_0 = 0$. Left: the dependence of $\tilde{\mathcal{E}}_L$ versus $\omega$ for $L = 2\pi,3\pi,4\pi$. Right: the spatial profile $\phi$ versus $x$ for $\omega = 0.3,0.6,0.9$ and $L = 4\pi$. }
	\label{fig-odd}
\end{figure}

By using the numerical approximation of the spatial profile $\phi$, 
we can also compute the mass $Q(\phi)$ for a fixed spatial period $L > 0$ 
and plot it versus $\omega$ to verify the sharp criterion for the energetic 
stability of the periodic waves given by Theorem \ref{th-stability}. Figure \ref{fig-mass} plots $Q(\phi)$ versus 
$\omega$ for $L = 2\pi,3\pi,4\pi$. The dashed line shows the dependence 
of $Q(\phi)$ in the limit $L \to \infty$, which corresponds to the 
solitary waves. The left panel presents the mapping $\omega \to Q(\phi)$ 
for the even wave satisfying (\ref{even-wave}) and the right panel 
presents the same for the odd wave satisfying (\ref{odd-wave}). The numerical 
inaccuracies occur near $\omega = 1$ and the end points of the numerical data are shown by solid dots. By using (\ref{phi-limeven}) and (\ref{phi-limodd}), we are able to compute $Q(\phi)$ analytically at 
$\omega = 1$ for the peaked waves, see (\ref{Q-even-limit}) and (\ref{Q-odd-limit}) below, and show the result in Figure \ref{fig-mass} by open dots. An interpolation between the right solid dot and the open dot is shown 
by dotted line.

\begin{figure}[!ht]
	\centering
	\includegraphics[width=0.94\textwidth,height=0.3\textheight]{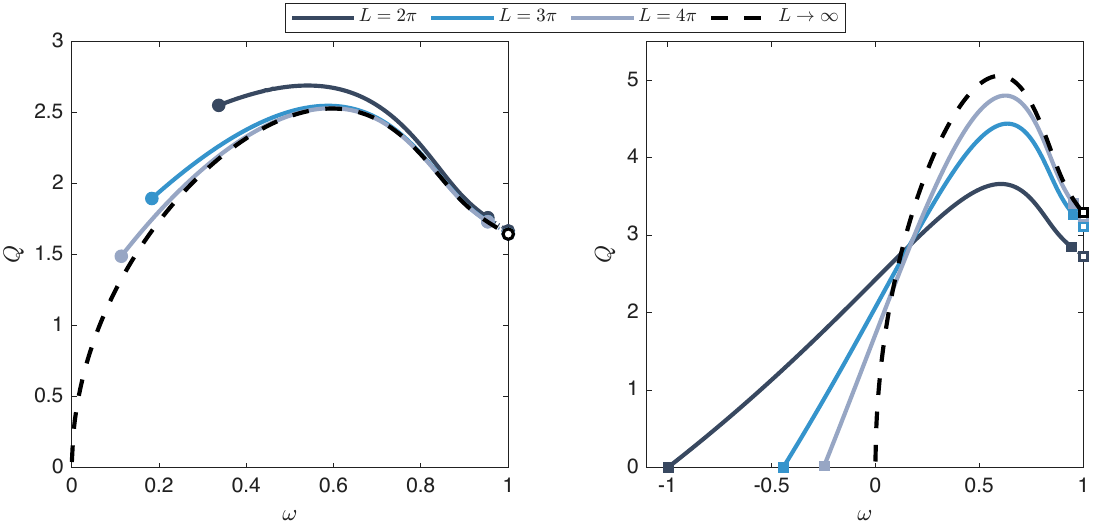} 
	\caption{Dependence of $Q(\phi)$ versus $\omega$ for $L = 2\pi,3\pi,4\pi$ and in the limit $L \to \infty$ (dashed line). Left panel: the even wave satisfying (\ref{even-wave}). Right panel: the odd wave satisfying (\ref{odd-wave}).}
	\label{fig-mass}
\end{figure}

Based on the numerical approximations and the sharp criterion in Theorem \ref{th-stability}, we conclude from Figure \ref{fig-mass} 
that both even and odd periodic waves are energetically stable for smaller 
values of $\omega$ and energetically unstable for values of $\omega$ near 
$\omega = 1$. To be precise, we formulate the following conjecture. 

\begin{conjecture}
There is $\omega_* \in (\omega_L,1)$ and $\Omega_* \in (0,1)$ such that 
the even wave satisfying (\ref{even-wave}) is energetically stable for 
$\omega \in (\omega_L,\omega_*)$ and unstable for $\omega \in (\omega_*,1)$, 
whereas the odd wave satisfying (\ref{odd-wave}) is energetically stable for 
$\omega \in (\Omega_L,\Omega_*)$ and unstable for $\omega \in (\Omega_*,1)$. 
\end{conjecture}

\begin{remark}
	\label{rem-data}
	The numerical data in Figures \ref{fig-period}, \ref{fig-even}, \ref{fig-odd}, and \ref{fig-mass} are obtained with high numerical accuracy, controlled within 
	$10^{-8}$ error, since the numerical error only arises in the computation of the period function $T(\mathcal{E},\omega)$ and the wave profile $\phi(x)$ from the corresponding integrals. The dependence of $Q(\phi)$ versus $\omega$ in the limit 
	$L \to \infty$ shown in Figure \ref{fig-mass} contradicts 
	the claim from \cite[Figure 5]{KPR24} that the dependence is monotonically increasing near $\omega = 0$ and $\omega = 1$ and decreasing for $\omega \in (\omega_1,\omega_2)$ for some $0 < \omega_1 < \omega_2 < 1$. Although the numerical data on $Q(\phi)$ versus $\omega$ in \cite{KPR24} was consistent with the numerical approximations of unstable eigenvalues in the spectral stability problem, see same Figure 5 in \cite{KPR24}, we have found that the claim 
	of stability of bright solitons near $\omega = 1$ in \cite{KPR24} is a numerical artefact. It is related with the center-difference approximations of the second-order derivatives with a large stepsize, which were used in \cite{KPR24}. By reducing the stepsize or performing computations with adaptive methods directly from (\ref{energy}) and (\ref{period-function}), we have found that $Q(\phi)$ is monotonically decreasing in $\omega$ near $\omega = 1$. Although our numerical data has a tiny gap near $\omega = 1$ due to the lack of numerical accuracy, comparison between the last numerical data (solid right dots) and the analytically computed limiting value of $Q(\phi)$ at $\omega = 1$ (open dots) suggest the monotone decrease of $Q(\phi)$ near $\omega = 1$.
\end{remark}

\subsection{Methodology and organization of the paper}
	
The existence of periodic orbits stated in Theorem \ref{th-existence} is obvious from the phase portrait shown in Figure \ref{fig-phase}. 
Nevertheless, we complement this dynamical system picture with the functional-analytic setup and prove the existence of periodic orbits based on the implicit function theorem. The family of even waves satisfying (\ref{even-wave}) 
with $x_0 = 0$ is studied in a subspace of the Sobolev space $H_{\rm per}^s$, $s \geq 0$ constituted by even periodic functions and denoted by $H_{\rm per,e}^s$. The family of odd waves satisfying (\ref{odd-wave}) with $x_0 = 0$ is studied in a subspace constituted by odd periodic functions and denoted by $H_{\rm per,o}^s$. These results are described in Section \ref{sec-2}.

The monotonicity of the period function stated in Theorem \ref{th-period} is proven with two different methods for the even and odd waves. For the even wave, 
we use Chicone's theorem \cite{chicone2} and confirm the monotonicity criterion 
based on the explicit analysis of the logarithmic and polynomial functions. 
For the odd wave, we estimate the period function by using convexity of the integrand functions. These results are described in Section \ref{sec-3}.

The energetic stability criterion stated in Theorem \ref{th-stability} 
is proven in two steps. 
As a first step, we analyze the Morse and nullity indices 
of the Schr\"{o}dinger operators $\mathcal{L}_{\pm} : H^2_{\rm per} \subset L^2_{\rm per} \to L^2_{\rm per}$ given by (\ref{matrixop}), where 
the Morse index denoted by $n(\mathcal{L}_{\pm})$ is the number of negative eigenvalues with the account of their multiplicities and the nullity index 
denoted by $z(\mathcal{L}_{\pm})$ is the multiplicity of the zero eigenvalue. 
For the even wave, we prove that $n(\mathcal{L}_+) = z(\mathcal{L}_+) = z(\mathcal{L}_-) = 1$ and $n(\mathcal{L}_-) = 0$. For the odd wave, 
we prove that $n(\mathcal{L}_+) = 2$, $n(\mathcal{L}_-) = z(\mathcal{L}_+) = z(\mathcal{L}_-) = 1$. These results are described in Section \ref{sec-4}.

As a second step, we analyze the Morse and nullity indices 
of the constrained operator $\mathcal{L}_+ |_{\{ \phi_0 \}^{\perp}}$, where 
the constraint with $\phi_0 \equiv \frac{\phi}{1- \phi^2}$ is due to the fixed mass $Q$ restriction \cite{GeyerPel25}. We show that $n(\mathcal{L}_+ |_{\{ \phi_0 \}^{\perp}}) = 0$ and $z(\mathcal{L}_+ |_{\{ \phi_0 \}^{\perp}}) = 1$ 
for the even wave if and only if the mapping $\omega \to Q(\phi)$ is monotonically 
increasing, which yields Theorem \ref{th-stability} for the even wave. 
We also show that $n(\mathcal{L}_+ |_{\{ \phi_0 \}^{\perp}}) = 1$ and $z(\mathcal{L}_+ |_{\{ \phi_0 \}^{\perp}}) = 1$ 
for the odd wave if and only if the mapping $\omega \to Q(\phi)$ is monotonically 
increasing. This is still inconclusive for the energetic stability of the odd wave. However, restricting $H^1_{\rm per}$ to the space $\mathcal{Y}$ of odd perturbations 
with respect to the half-period allows us to obtain 
$n(\mathcal{L}_+ |_{\{ \phi_0 \}^{\perp} \cap \mathcal{Y}}) = 
n(\mathcal{L}_- |_{\mathcal{Y}}) = 0$ and 
$z(\mathcal{L}_+ |_{\{ \phi_0 \}^{\perp} \cap \mathcal{Y}}) = 
z(\mathcal{L}_- |_{\mathcal{Y}}) = 1$ 
if and only if the mapping $\omega \to Q(\phi)$ is monotonically 
increasing, which yields Theorem \ref{th-stability} for the odd wave. These results are described in Section \ref{sec-5}. We note that the idea of restricting the space of periodic functions to odd perturbations with respect to the half-period is proposed in \cite{GalHar} for the stability analysis of odd waves in the cubic NLS equation.

Finally, the numerical methods used to compute the data in 
Figures \ref{fig-period}, \ref{fig-even}, \ref{fig-odd}, and \ref{fig-mass} 
are described in Section \ref{sec-6}.  We also elaborate Remark \ref{rem-data} about the limit $L \to \infty$ with more details.

\section{Existence of periodic waves} 
\label{sec-2}

We prove Theorem \ref{th-existence} within a functional analysis framework. 
Section \ref{sec-2-1} defines the basic facts used in the proofs. 
Sections \ref{sec-2-2} and \ref{sec-2-3} provide global 
continuations of the even and odd waves
for the fixed spatial period $L > 0$.

\subsection{Preliminary facts}
\label{sec-2-1}

We first recall some facts regarding Fredholm operators on a Banach space $X$. 
An unbounded operator $S : D(S) \subset X\rightarrow X$ is called {\em a Fredholm operator} if ${\rm{Range}}(S)$ is closed and ${\rm{dim}}(\Ker(S))$ and ${\rm{dim}}({\rm{Coker}}(S))$ are both finite, where ${\rm{Coker}}(S)$ denotes the quotient space given by ${\rm{Coker}}(S)=\bigslant{X}{{\rm Range}(S)}$.
{\em The index} of an unbounded Fredholm operator $S:D(S)\subset X\rightarrow X$ is given by 
$$
{\rm ind}(S)={\rm{dim}}(\Ker(S))-{\rm{dim}}({\rm{Coker}}(S))\in \mathbb{Z}.
$$ 
A Fredholm operator is of index zero if ${\rm ind}(S)=0$.		

The following lemma provides a result that is useful for our purposes since we obtain a suitable relation between $c(S)={\rm{dim}}({\rm{Coker}}(S))$ and $z(S)=\dim(\Ker(S))$.
	
	\begin{lemma}
		\label{lem-Hilbert} 
		Let $H$ be a real Hilbert space and let $K\subset H$ be a closed subspace. Then, $\bigslant{H}{K}\cong K^{\bot}$.
	\end{lemma}		
		
	\begin{proof} 
		Let us define $\mathcal{T}:\bigslant{H}{K}\rightarrow K^{\bot}$ given by $\mathcal{T}(u)=u-P_Ku$ for any $u \in H$, where $P_K$ is the orthogonal projection from $H$ onto the closed subspace $K$. Since $||\mathcal{T}(u)||_H=||u-P_Ku||$, we obtain by the Pythagorean theorem  
		$$
		||u||_{H}^2 = ||P_K u||_H^2 + ||u-P_Ku||_H^2 = ||P_Ku||_H^2 + ||\mathcal{T}(u)||_H^2.
		$$
		The equality implies $||\mathcal{T}(u)||_H^2=||u||_{H}^2-||P_Ku||_H^2\leq ||u||_{H}^2$, and thus, $\mathcal{T}$ is a bounded operator. $\mathcal{T}$ is a one-to-one operator since $T(u)=0$ implies $u=P_Ku$ and this fact automatically implies $u \in K$. To show that $\mathcal{T}$ is onto, we consider $v\in K^{\bot}$. By the definition of orthogonal projection from $H$ onto the closed subspace $K$, there exists $u\in H$ such that $v=u-P_Ku$, and $\mathcal{T}$ is onto as desired.	
	\end{proof}
	
	\begin{remark}\label{rem1} 
		If $S:D(S)\subset H\rightarrow H$ is an unbounded self-adjoint linear operator with a closed range on a Hilbert space $H$, then we have by Lemma $\ref{lem-Hilbert}$ that 
		$$
		\bigslant{H}{{\rm Range}(S)}= \bigslant{H}{\Ker(S)^{\bot}}\cong\Ker(S)^{{\bot}{\bot}}=\Ker(S).
		$$ 
		Therefore, if ${\rm dim}(\Ker(S))$ is finite, then $S$ is always a Fredholm operator of index zero. 
	\end{remark}
	
	We conclude this section by stating a version of the implicit 
	function theorem used in our study, see \cite[Theorem 8.3.1]{buffoni-toland}. 
	
	\begin{theorem}\label{TCR} 
		Suppose that $X$ and $Y$ are Banach spaces, that $\mathsf{F}: X\times \mathbb{R}\rightarrow Y$ is of class $C^k$, $k\geq2$, and that $\mathsf{F}(x_0,\lambda_0)=0\in Y$ for some $(x_0,\lambda_0) \in X\times \mathbb{R}$. Suppose also that
		\begin{enumerate}
			\item $\partial_g\mathsf{F}(g,\lambda)$ is a Fredholm operator of index zero when $\mathsf{F}(g,\lambda)=0$ for all $(g,\lambda)\in U$. Here $U\subset X\times\mathbb{R}$ denotes an open subset.
			\item For some $(x_0,\lambda_0)\in X\times\mathbb{R}$, $\ker(L_{\lambda_0})$ is one dimensional, where $L_{\lambda_0}=\partial_g\mathsf{F}(x_0,\lambda_0)$. This means that
			$\ker(L_{\lambda_0})=\{h\in X;\ h=sh_0\ \mbox{for some}\ s\in\mathbb{R}\},$ $h_0\in X\backslash\{0\}$.
			\item The transversality condition holds: $\partial_{\lambda,g}^2\mathsf{F}[(x_0,\lambda_0)](1,h_0)\notin {\rm{Range}}(L_{\lambda_0})$.
		\end{enumerate} 
		Then, there exists $a_0>0$ and a branch of solutions 
		$\{(\chi(a),\Lambda(a));\ a\in (-a_0,a_0) \}\subset X\times\mathbb{R},$
		such that $\Lambda(0)=\lambda_0$ and $\chi(0)=x_0$. In addition, we have
		\begin{itemize}
			\item $\mathsf{F}(\chi(a),\Lambda(a))=0$ for all $a\in (-a_0,a_0)$.
			\item $a\mapsto\Lambda(a)$ and $a\mapsto \chi(a)$ are of class $C^{k-1}$ on $(-a_0,a_0)$.
			\item there exists an open set $U_0\subset  X\times \mathbb{R}$ such that $(x_0,\lambda_0)\in U_0$ and 
			$$\{(g,\lambda)\in U_0;\ \mathsf{F}(g,\lambda)=0,\ g\neq0\}=\{(\chi(a),\Lambda(a));\ a \in (-a_0,a_0) \}.
			$$
			\item If $\mathsf{F}$ is analytic, then  $\chi$ and $\Lambda$ are analytic functions on $(-a_0,a_0)$.					
		\end{itemize}
	\end{theorem}

	\subsection{Existence of (positive) even periodic waves }
	\label{sec-2-2}

	To prove the existence of even periodic waves, we consider the subspace $L_{\rm per,e}^2$ contained in $L_{\rm per}^2$, which consists of even periodic functions, that is, functions in the Hilbert space
	$$
	L_{\rm per,e}^2 = \left\{ f\in L_{\rm per}^2 : \quad f(-x)=f(x) \;\; {\rm a.e.} \;\; x \in \left[-\frac{L}{2},\frac{L}{2}\right] \right\},
	$$
	where the spatial period $L > 0$ is fixed. The first result establishes a local 
	bifurcation of small (positive) even periodic waves from the (positive) constant solution $\phi = \sqrt{\omega}$ of the second-order equation (\ref{ode}).
	
	\begin{proposition}\label{prop1}
		There exists $a_0>0$ such that for all $a\in (-a_0,a_0)$ there exists an even periodic solution $\phi \in H^2_{\rm per,e}$ to the second-order equation (\ref{ode}) given by  
		\begin{equation}\label{smallsol}
			\phi(x)=\sqrt{\omega}+\sum_{n=1}^{\infty}\phi_n(x)a^n,
		\end{equation}
		where $\{ \phi_n \}_{n \in \mathbb{N}}$ are uniquely determined in $H^2_{\rm per,e}$. The frequency $\omega$ of the $L$-periodic wave is given by
		\begin{equation}\label{freqwave}
			\omega = \omega_L + \sum_{n=1}^{\infty} \omega_{2n} a^{2n},
		\end{equation}
		where $\omega_L = \frac{2 \pi^2}{L^2 + 2 \pi^2}$ and $\{ \omega_{2n} \}_{n \in \mathbb{N}}$ are uniquely determined constants. Furthermore, we have $\omega > \omega_L$ for small $a \neq 0$.
	\end{proposition}

	\begin{proof}
		We outline the steps used to prove the existence of small--amplitude periodic waves using Theorem $\ref{TCR}$. Let $\mathsf{F}:H_{\rm per,e}^2\times\mathbb{R}\rightarrow L_{\rm per,e}^2$ be the smooth map defined by
		$$
			\mathsf{F}(g,\lambda)=-(1-g^2)g''+\lambda g-g^3.
		$$
			We see that $\mathsf{F}(g,\lambda)=0$ if and only if $g\in H_{\rm per,e}^2$ satisfies $(\ref{ode})$ with corresponding frequency $\lambda = \omega \in \R$. Let $\lambda_0>0$ be fixed. The Fr\'echet derivative of the function $\mathsf{F}$ with respect to the first variable at $(\sqrt{\lambda_0},\lambda_0)$ is given by
		$$
			D_{g}\mathsf{F}(\sqrt{\lambda_0},\lambda_0)f=-(1-\lambda_0)f''-2\lambda_0f.
	$$
The nontrivial kernel of $	D_{g}\mathsf{F}(\sqrt{\lambda_0},\lambda_0)$ is determined by functions $h\in H_{\rm per,e}^2$ such that 
$$
\widehat{h}(n)\left((1-\lambda_0)\left(\frac{2\pi n}{L}\right)^2-2\lambda_0\right)=0, \quad n\in\mathbb{N}\backslash\{0\},
$$
where $\widehat{h}(n)$ is the $n$-th coefficient of the Fourier series of $h\in L_{\rm per}^2$. We see that 	$D_{g}\mathsf{F}(\sqrt{\lambda_0},\lambda_0)$ has nontrivial kernel if, and only if, $\lambda_0=\frac{\left(\frac{2\pi n}{L}\right)^2}{2+\left(\frac{2\pi n}{L}\right)^2}>0$ for some $n\in\mathbb{N}\backslash\{0\}$. In this case, we have 
$$	
\Ker D_{g}\mathsf{F}(\sqrt{\lambda_0},\lambda_0) = {\rm Span}(\widetilde{\varphi}_n), \quad \widetilde{\varphi}_n(x)=\cos\left(\frac{2\pi n}{L}x\right), \quad n\in\mathbb{N}\backslash\{0\}.
$$
In addition, since $D_{g}\mathsf{F}(\sqrt{\lambda_0},\lambda_0)$ is a self-adjoint operator defined in $L_{\rm per,e}^2$ with domain $H_{\rm per,e}^2$, the transversality condition is also satisfied:
$$
\cos\left(\frac{2\pi n}{L}x\right) \notin \Ker D_{g}\mathsf{F}(\sqrt{\lambda_0},\lambda_0)^{\bot} = {\rm Range}D_{g}\mathsf{F}(\sqrt{\lambda_0},\lambda_0).
$$ 

To obtain periodic solutions with minimal spatial period $L > 0$, we must consider $n=1$. Thus, we have $\lambda_0 = \frac{2 \pi^2}{L^2 + 2 \pi^2} = \omega_L$ and define the set $\mathcal{S}$ given by 
		$$
		\mathcal{S}=\{(g,\lambda)\in U;\ \mathsf{F}(g,\lambda)=0\},
		$$
		where 
			$$
			U=\left\{(g,\lambda)\in H_{per,e}^2\times \left(\lambda_0,1\right);\ 0<g<1\right\}.
			$$			
Let $(g, \lambda) \in \mathcal{S}$ be a solution of the equation $\mathsf{F}(g, \lambda) = 0$. First, we prove that the linear operator
$$
D_gF(g,\lambda)=-(1-g^2)\partial_x^2+\lambda-3g^2+2gg''
$$ 
is a Fredholm operator of index zero. In fact, in order to simplify the notation, let us denote 
$$
Q = D_gF(g,\lambda) \quad \mbox{\rm and} \quad P=(1-g^2)^{-1}Q.
$$ 
First, $P$ is clearly a self-adjoint operator. Thus,
$\sigma(P)=\sigma_{\rm disc}(P)\cup \sigma_{\rm ess}(P)$,  where $\sigma(P)$
denotes the spectrum of  $P$, and  $\sigma_{\rm disc}(P)$ and $\sigma_{\rm ess}(P)$
denote, respectively, the discrete and essential spectra of  $P$. Since $H^2_{\rm per,e}$ is compactly embedded in $L^2_{\rm per,e}$, the operator $P$ has compact resolvent. Consequently, $\sigma_{\rm ess}(P) = \emptyset$, and $\sigma(P)=\sigma_{\rm disc}(P)$ consists of isolated eigenvalues with finite algebraic multiplicities (see e.g., \cite[Section III.6]{kato}). Since $(g,\lambda)\in \mathcal{S}$, we see that $0$ is an eigenvalue for $P$ associated with the eigenfunction $g'$ and $z(P)$ is finite. A basis for the subspace $\Ker(P)$  can be taken as $\{v_1,\cdots,v_n\}$. 

On one hand, by Remark $\ref{rem1}$ we have 
$$
\bigslant{L_{\rm per,e}^2}{{\rm Range}(Q)}\cong {\rm Range}(Q)^{\bot}=\Ker(Q^{*})=\Ker(P(1-g^2)).
$$ 
Since $\{v_1(1-g^2)^{-1},\cdots,v_n(1-g^2)^{-1}\}$ is a basis for the subspace $\Ker(Q^{*})$, it follows that $z(P)=z(Q^{*})$, and $z(Q^{*})$ is finite. On the other hand, we have $\Ker(P)=\Ker(Q)$, so that $z(Q^{*})=z(P)=z(Q)$ and the index of the Fredholm operator $Q$ is zero as desired.

Thus, by defining $\lambda_0=\omega_L \in (0,1)$, the local bifurcation theory established in Theorem $\ref{TCR}$ guarantees the existence of an open interval $I\subset (0,1)$ near $\omega_L$, an open ball $B(\sqrt{\omega_L},r) \subset H_{\rm per,e}^{2}$, around the equilibrium solution $\sqrt{\omega_L}$, for some $r>0$ and a  smooth mapping
			\begin{equation}
			\label{localcurve}\omega \in I \mapsto \phi \in B(\sqrt{\omega_L},r) \subset H_{\rm per,e}^{2},
			\end{equation}
			such that $\mathsf{F}(\phi,\omega) = 0$ for all $\omega \in I$ and $\phi\in B(\sqrt{\omega_L},r)$.

Next, we determine the first terms in the expansions $(\ref{smallsol})$ and $(\ref{freqwave})$. To simplify the notation, let us define $s=\sqrt{\omega_L}$. The Taylor expansion of the square-root function at $\omega_L$ yields 
			$$
				\sqrt{\omega}=s+\frac{\omega_2}{2s}a^2+\mathcal{O}(a^4).
				$$
The correction terms $\omega_2$ and $\{ \phi_1,\phi_2,\phi_3\}$ are uniquely determined by the following recurrence relations
			 \begin{equation}\label{recurrence}\left\{\begin{array}{llllll}
					\mathcal{O}(a)\ : &-(1-s^2)\phi_1''-2s^2\phi_1=0,\\
					\mathcal{O}(a^2):&-(1-s^2)\phi_2''-2s^2\phi_2+2s\phi_1\phi_1''-3s\phi_1^2=0,\\
					\mathcal{O}(a^3): &-(1-s^2)\phi_3''-2s^2\phi_3+\omega_2(\phi_1''-2\phi_1)+2s(\phi_1\phi_2''+\phi_2\phi_1'')\\
					& \qquad -6s\phi_1\phi_2-\phi_1^3+\phi_1^2\phi_1''=0.
				\end{array}\right.\end{equation}
We see that $\phi_1(x)=\cos(kx)$ satisfies the equation containing the term $\mathcal{O}(a)$ for $k=\frac{2\pi}{L}$ since $\omega_L= \frac{2\pi^2}{L^2 + 2 \pi^2}$. Solving the inhomogeneous equation for $\mathcal{O}(a^2)$, we obtain
			$$
			\phi_2(x)=\frac{s^2+3}{12s(1-s^2)}(\cos(2kx)-3).
			$$
			We need to find the constant $\omega_2$ in the third equation of $(\ref{recurrence})$. The inhomogeneous equation at $\mathcal{O}(a^3)$ admits a solution $\phi_3 \in H^2_{\rm per,e}$ if, and only if, the right-hand side is orthogonal to $\phi_1$, which selects uniquely the
			correction 
			$$
			\omega_2 = \frac{s^4+6s^2-9}{6(s^2-1)} = \frac{9 - 6 \omega_L - \omega_L^2}{6(1-\omega_L)}.
			$$ 
Since $\omega_2 > 0$, the solution (\ref{smallsol}) with (\ref{freqwave}) exists for $\omega > \omega_L$ near $\omega_L$. This finishes the proof of the proposition.
	\end{proof}
	
	\begin{remark}
The bifurcating solution obtained in Proposition $\ref{prop1}$ is unique in $H^2_{\rm per,e}$, up to the parametrization provided by the bifurcation parameter. This uniqueness is ensured since the Lyapunov–Schmidt reduction requires the application of the implicit function theorem. In the case of a one-dimensional kernel, the bifurcation occurs along a single branch of solutions. The implicit function theorem then guarantees the existence of a unique smooth curve of solutions that bifurcates from the constant solution.
		\end{remark}
	
The next result establishes a global continuation of (positive) even periodic waves from the local bifurcating solution obtained in Proposition \ref{prop1}.
	
	\begin{proposition}\label{prop2}
		The local solution obtained in Proposition $\ref{prop1}$ is global, that is, $\phi$ exists for all $\omega\in \left(\omega_L,1\right)$. In addition, the pair $(\phi,\omega)\in H_{\rm per,e}^2\times\left(\omega_L,1\right)$ is continuous with respect to the parameter $\omega\in \left(\omega_L,1\right)$ and it satisfies $(\ref{ode})$.
		\end{proposition}
	
		\begin{proof}
			To extend the local curve in $(\ref{localcurve})$ to a global curve, we need to prove that every bounded and closed subset $\mathcal{R}\subset \mathcal{S}$ is a compact set contained in $H_{\rm per,e}^2\times \left(\omega_L,1\right)$. To this end, we want to prove that $\mathcal{R}$ is sequentially compact, that is, if $\{(\phi_m,\omega_m)\}_{m\in\mathbb{N}}$ is a sequence in $\mathcal{R}$, there exists a subsequence of $\{(\phi_m,\omega_m)\}_{m\in\mathbb{N}}$ that converges to a point in $\mathcal{R}$. Up to a subsequence, we obtain 
			\begin{equation}\label{omegam}
				\omega_m\rightarrow \omega\ \ \mbox{in}\ \left[\omega_L,1\right],
				\end{equation}
				and 
			\begin{equation}\label{phim}
				\phi_m\rightharpoonup \phi\ \ \mbox{in}\ H_{\rm per,e}^2,
			\end{equation}	
			as $m\rightarrow +\infty$. Next, $\{\phi_m\}$ is a bounded sequence in $H_{\rm per,e}^2$ and it satisfies $0<\phi_m<1$. Since the embedding $H_{\rm per,e}^2\hookrightarrow H_{\rm per,e}^1$ is compact and $H_{\rm per,e}^1$ is a Banach algebra, we obtain $\phi_m^3\rightarrow \phi^3$ in $H_{\rm per,e}^1$ as $m \rightarrow +\infty$. In particular, we have $\phi_m^2\rightarrow \phi^2$ in $H_{\rm per,e}^1\hookrightarrow C_{\rm per,e}$ and by $(\ref{phim})$, we obtain
			\begin{equation}\label{phim1}
				(1-\phi_m^2)\phi_m''\rightharpoonup (1-\phi^2)\phi''\ \ \mbox{in}\ L_{\rm per,e}^2,
			\end{equation}
			as $m \rightarrow +\infty$. Since in particular $\phi_m^3\rightarrow \phi^3$ in $L_{\rm per,e}^2$ as $m \rightarrow +\infty$, we obtain by $(\ref{omegam})$ and $(\ref{phim1})$ that the pair $(\phi,\omega)\in H_{\rm per,e}^2\times \left[\omega_L,1\right]$ satisfies
			\begin{equation}\label{odeglobal}
			-(1-\phi^2)\phi''+\omega\phi-\phi^3=0.
		\end{equation}	
We see from $(\ref{odeglobal})$ that $\omega < 1$, since smooth periodic solutions to this equation do not exist when $\omega = 1$. Moreover, if $\omega = \omega_L$, the constant solution $\phi = \sqrt{\omega_L}$ is the only solution to $(\ref{odeglobal})$. Thus $\omega\in \left(\omega_L,1\right)$ as requested and we have $\phi > 0$. On the other hand, if there exists $t_0\in [0,L]$ such that $\lim_{t\rightarrow t_0}\phi(t)=1$, then $\omega=1$, which is a contradiction, since in this case there are no periodic solutions of $(\ref{odeglobal})$. Therefore, we obtain $0<\phi<1$ as requested. 

Finally, by $(\ref{freqwave})$ the frequency $\omega$ of the periodic wave is not constant. By applying \cite[Theorem 9.1.1]{buffoni-toland} we can extend globally the local bifurcation curve given in $(\ref{localcurve})$. More precisely, there is a continuous mapping
		$$
		\omega \in \left(\omega_L,1\right) \mapsto \phi \in  H_{\rm per,e}^{2},
		$$
		where $\phi$ solves equation $(\ref{odeglobal})$.
		\end{proof}

\begin{remark}\label{boundphi}
Since $\phi$ is continuous, satisfies $0 < \phi < 1$, and there is no $t_{\pm} \in [0, L]$ such that 
$\lim_{t \to t_-} \phi(t) = 0$ and $\lim_{t \to t_+} \phi(t) = 1$, we obtain that there exist $m$ and $M$ that depend on $\omega \in (\omega_L,1)$ such that $0 < m < M<1$ and $m < \phi(x) \leq M$ for every $x \in [0, L]$. In fact, since $\phi$ is continuous on the compact set $[0, L]$, we have 
$$
m = \min_{[0,L]} \phi(x), \quad M =\max_{[0,L]}\phi(x).
$$ 
\end{remark}	

Propositions \ref{prop1} and \ref{prop2}, as well as Remark \ref{boundphi}, justify the existence result stated in Theorem \ref{th-existence} for the even wave satisfying (\ref{even-wave}) with $x_0 = 0$.

\subsection{Existence of odd periodic waves}
		\label{sec-2-3}

		To prove the existence of odd periodic waves, we consider the subspace $L_{\rm per,o}^2$ contained in $L_{\rm per}^2$, which consists of odd periodic functions, that is, 
		$$
		L_{\rm per,o}^2 = \left\{ f\in L_{\rm per}^2 : \quad f(-x)=-f(x)\ {\rm a.e.} \;\; x \in \left[-\frac{L}{2},\frac{L}{2}\right] \right\},
		$$
		where the spatial period $L > 0$ is fixed. The first result gives a local bifurcation of small odd periodic waves from the zero solution $\phi = 0$ of the second-order equation (\ref{ode}).
		
		\begin{proposition}\label{prop11}
			There exists $a_0>0$ such that for all $a\in (-a_0,a_0)$ there exists an odd periodic solution $\phi \in H^2_{\rm per,o}$ to the second-order equation $(\ref{ode})$ given by
			\begin{equation}
			\label{smallsol1}
				\phi(x) = \sum_{n=1}^{\infty}\phi_{2n-1}(x)a^{2n-1},
			\end{equation}
			where $\{ \phi_n \}_{n \in \mathbb{N}}$ are uniquely determined functions in $L^2_{\rm per,o}$. The frequency $\omega$ of the $L$-periodic wave is given by
			\begin{equation}
			\label{freqwave1}
				\omega = \Omega_L +\sum_{n=1}^{\infty} \omega_{2n} a^{2n},
			\end{equation}
			where $\Omega_L = -\frac{4\pi^2}{L^2}$ and $\{ \omega_{2n} \}_{n \in \mathbb{N}}$ are uniquely determined constants. Furthermore, we have $\omega > \Omega_L$ for small $a \neq 0$.
		\end{proposition}
	
		\begin{proof} 
			The proof is similar to that of Proposition $\ref{prop1}$ and therefore, we only outline the main steps. Indeed, let $\mathsf{G}:H_{\rm per,o}^2 \times\mathbb{R} \rightarrow L_{\rm per,o}^2$ be the smooth map defined by
			$$
				\mathsf{G}(g,\lambda)=-(1-g^2)g''+\lambda g-g^3.
			$$
			We see that $\mathsf{G}(g,\lambda)=0$ if, and only if, $g\in H_{\rm per,o}^2$ satisfies $(\ref{ode})$ with corresponding frequency $\lambda = \omega \in \R$. Let $\lambda_0\in\mathbb{R}$ be fixed. The Fr\'echet derivative of the function $\mathsf{G}$ with respect to the first variable at $(0,\lambda_0)$ is then given by
			$$
				D_{g}\mathsf{G}(0,\lambda_0)f=-f''+\lambda_0f.
			$$
The nontrivial kernel of $	D_{g}\mathsf{G}(0,\lambda_0)$ is determined by functions $h\in H_{\rm per,o}^2$ such that 
$$
\widehat{h}(n)\left(-\left(\frac{2\pi n}{L}\right)^2+\lambda_0\right)=0, \quad 
n \in\mathbb{N}\backslash\{0\}.
$$ 
Since $D_{g}\mathsf{G}(0,\lambda_0)$ has nontrivial kernel if and only if $\lambda_0=-\left(\frac{2\pi n}{L}\right)^2<0$ for some $n\in\mathbb{N}\backslash\{0\}$, we have
$$	
\Ker D_{g}\mathsf{G}(0,\lambda_0) =[\widetilde{\varphi}_n], \quad \widetilde{\varphi}_n(x)=\sin\left(\frac{2\pi n}{L}x\right), 
\quad n\in\mathbb{N}\backslash\{0\}.
$$	
To obtain periodic solutions with the minimal spatial period $L > 0$, we must consider $n=1$. Thus, we have $\lambda_0 = -\frac{4 \pi^2}{L^2} = \Omega_L$. 
		
The remainder of the proof is identical to the one in Proposition $\ref{prop1}$ but in order to complete the proof, we shall compute the first terms in the expansions $(\ref{smallsol1})$ and $(\ref{freqwave1})$. Indeed, for $k=\frac{2\pi}{L}$ the corrections terms $\omega_2$ and $\{ \phi_1, \phi_3\}$ are uniquely determined by the following recurrence relations
\begin{equation}
\label{recurrence1}
\left\{\begin{array}{llllll}
\mathcal{O}(a)\ : &-\phi_1''-k^2\phi_1=0,\\
\mathcal{O}(a^3): &-\phi_3''-k^2\phi_3+\omega_2\phi_1+\phi_1^2\phi_1''-\phi_1^3=0.
\end{array}\right.
\end{equation}
We see that $\phi_1(x)=\sin(kx)$ satisfies the equation containing the term $\mathcal{O}(a)$. To find the constant $\omega_2$ in the second equation of $(\ref{recurrence1})$, we use the fact that the inhomogeneous equation at $\mathcal{O}(a^3)$ admits a solution $\phi_3 \in H^2_{\rm per,o}$ if, and only if, the right-hand side is orthogonal to $\phi_1$, which selects uniquely the
correction 
$$
\omega_2=\frac{3}{4}(1+k^2).
$$
Since $\omega_2 > 0$, the solution (\ref{smallsol1}) with (\ref{freqwave1}) 
exists for $\omega > \Omega_L$ near $\Omega_L$. Finally, we solve the inhomogeneous equation for $\mathcal{O}(a^3)$ and obtain 
$$
\phi_3(x)=-\frac{1+k^2}{32k^2}\sin(3kx).
$$
This finishes the proof of the proposition.
\end{proof}
			
\begin{remark}\label{remodd}
Using analogous computations as in Proposition $\ref{prop2}$, we can establish that the local solution obtained in Proposition $\ref{prop11}$ is global, that is, $\phi$ exists for all $\omega \in (\Omega_L,1)$. In addition, the pair $(\phi,\omega)\in H_{\rm per,o}^2\times (\Omega_L,1)$ is continuous with respect to the parameter $\omega \in (\Omega_L,1)$ and it satisfies $(\ref{ode})$. 
Furthermore, there exists $M$ that depends on $\omega \in (\Omega_L,1)$
such that $-1 < -M < 0 < M < 1$ and 
$$
M =\max_{[0,L]}\phi(x) = -\min_{[0,L]} \phi(x).
$$ 
\end{remark}

Proposition \ref{prop11} and Remark \ref{remodd} justify the existence result stated in Theorem \ref{th-existence} for the odd wave satisfying (\ref{odd-wave}) with $x_0 = 0$.
			
\section{Monotonicity of the period function}
\label{sec-3}

We prove Theorem \ref{th-period} by analysing the period function $T = T(\mathcal{E},\omega)$ introduced in (\ref{period-function}). The period function is associated with the periodic orbits on the phase plane for the system of ordinary differential equations
\begin{equation}
\label{sysode}
\left\{\begin{array}{lllll}
\phi' = \xi, \\\\
\xi'=\frac{\omega \phi}{1-\phi^2}-\frac{\phi^3}{1-\phi^2}.
\end{array}\right.
\end{equation}
It follows from the theory of ordinary differential equations that the solution $\phi$ depends smoothly on the parameter $\mathcal{E} = E(\phi,\phi')$,
where the energy function is 
\begin{equation}
\label{energy-again}
E(\phi,\phi') = \frac{1}{2} (\phi')^2 + V(\phi), \quad V(\phi) = \frac{1}{2} (\omega - \phi^2) + \frac{1}{2} (1-\omega) \log \frac{1-\omega}{1-\phi^2}.
\end{equation}
For $\omega \in (0,1)$, the even wave satisfying (\ref{even-wave}) corresponds to $\mathcal{E} \in (0,\mathcal{E}_{\omega})$ and the odd wave satisfying (\ref{odd-wave}) 
corresponds to $\mathcal{E} \in (\mathcal{E}_{\omega},\infty)$, 
where $\mathcal{E}_{\omega} = E(0,0)$ corresponds to the energy level 
of the pair of homoclinic orbits from the saddle point $(0,0)$ 
which surround the center points $(\pm \sqrt{\omega},0)$. 
We note that 
$$
V(\pm \sqrt{\omega}) = 0 \quad \mbox{\rm and} \quad 
\lim_{\phi \to \pm 1} V(\phi) = +\infty.
$$ 
Furthermore, $V(\phi) \geq 0$ for all $\phi  \in [-1,1]$ and $V'(\phi) > 0$ for all $\phi \in (\sqrt{\omega},1)$.

Section \ref{sec-3-1} gives the proof of $\partial_{\mathcal{E}} T(\mathcal{E},\omega) > 0$, $\mathcal{E} \in (0,\mathcal{E}_{\omega})$ for 
the periodic orbits inside the homoclinic orbit (the even waves). 
Section \ref{sec-3-2} gives the proof of $\partial_{\mathcal{E}} T(\mathcal{E},\omega) < 0$, $\mathcal{E} \in (\mathcal{E}_{\omega},\infty)$ for 
the periodic orbits outside the pair of homoclinic orbits (the odd waves). 
The latter result also holds for $\omega \in (-\infty,0)$ and $\mathcal{E} \in (\mathcal{E}_{\omega},\infty)$, for which $\mathcal{E}_{\omega} = E(0,0)$ corresponds to the energy level of the center point $(0,0)$.

\subsection{Monotonicity for even periodic waves}
\label{sec-3-1}

By the main theorem in \cite{chicone2}, the period function $T(\mathcal{E},\omega)$ is monotonically increasing in $\mathcal{E}$ in $(0,\mathcal{E}_{\omega})$ if $I''(\phi) > 0$ for $\phi \in (0,1)$, where 
\begin{equation}
\label{I}
I(\phi) = \frac{V(\phi)}{[V'(\phi)]^2}.
\end{equation}
Note that the theorem in \cite{chicone2} can be applied because $V(\sqrt{\omega}) = 0$ is properly normalized at the center point $(\sqrt{\omega},0)$. Computing
\begin{equation*}
V'(\phi) = -\frac{\phi (\omega-\phi^2)}{1-\phi^2} , \quad V''(\phi) = -\frac{\omega + (\omega-3) \phi^2 + \phi^4}{(1-\phi^2)^2} , \quad V'''(\phi) =  \frac{ 2 (1-\omega) \phi (\phi^2 + 3)}{(1-\phi^2)^3},
\end{equation*}
we obtain from (\ref{I}) that 
\begin{equation}
\label{I-second}
I''(\phi) = \frac{6V(V'')^2 - 2VV'V'''-3(V')^2V''}{(V')^4} =: \frac{P(\phi)}{(V'(\phi))^4 (1 -\phi^2)^4},
\end{equation}
where
\begin{equation*}
P(\phi) = 3 \phi^2  (\omega - \phi^2)^2 A(\phi) + \left[ 3 A(\phi)^2 + 2 (1-\omega) \phi^2 (3 + \phi^2) (\omega - \phi^2) \right] B(\phi),
\end{equation*}
with
\begin{equation*}
A(\phi) := \omega + (\omega-3) \phi^2 + \phi^4, \qquad B(\phi) := \omega-\phi^2 + (1 - \omega) \log \frac{1-\omega}{1 - \phi^2}.
\end{equation*}
Since $P$ depends on $\phi^2$, we introduce $t = \phi^2$ and redefine 
$P$, $A$, and $B$ as functions of $t$:
\begin{equation}
\label{N}
P(t) = 3 t (\omega - t)^2 A(t) + \left[  3A(t)^2 +  2  (1-\omega) t (3 + t) (\omega - t)  \right ] B(t), \quad t\in [0,1),
\end{equation}
with 
\begin{equation}\label{AB}
A(t) = \omega + (\omega-3) t + t^2, \qquad 
B(t) = \omega-t + (1 - \omega) \log \frac{1-\omega}{1 - t}.
\end{equation}
Figure \ref{fig-plot-P} shows the dependence of $P$ versus $t$ for $\omega = 0.5$. The plot suggests that 
\begin{itemize}
	\item $P(t)$ has a quadruple zero at $t = \omega$, 
	\item $P(t) > 0$ for $t \in (0,\omega) \cup (\omega,1)$. 
\end{itemize}
These facts are proven rigorously in Lemmas \ref{lemma1} and \ref{lemma3} below.

\begin{figure}[htb!]
	\centering
	\includegraphics[width=0.8\textwidth,height=0.4\textheight]{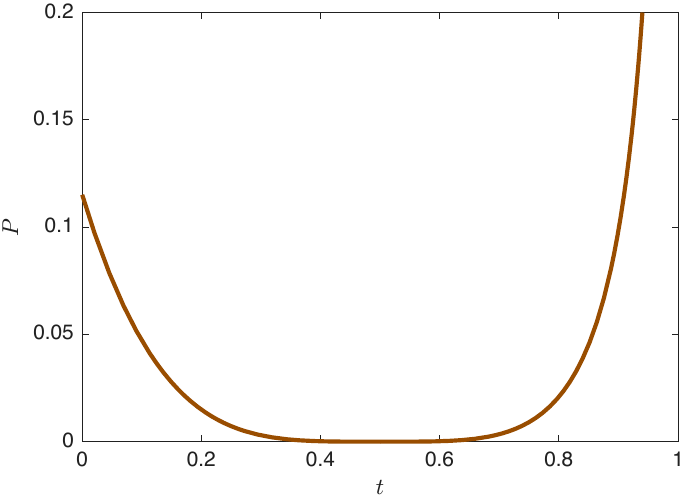} 
	\caption{The dependence of $P$ versus $t$ given by (\ref{N}) for $\omega = 0.5$.}
	\label{fig-plot-P}
\end{figure}

\begin{lemma}\label{lemma1}
	The function $P$ given by (\ref{N}) and (\ref{AB}) is real analytic on $(0,1)$ and it admits a zero of the quadruple order at $t=\omega$, such that
	\begin{equation}
	\label{Taylor-coeff}
	P'(\omega) = P''(\omega) = P'''(\omega) = 0, \quad \quad P^{(4)}(\omega) = \frac{4(9-6\omega-\omega^2)}{1-\omega} >0.
	\end{equation}
	Consequently, there exists $\delta>0$ such that $P(t)>0$ for $t\in [\omega-\delta,\omega +\delta] \backslash \{\omega \}$.
\end{lemma}

	\begin{proof}
		The function $P(t)$ is real analytic on $(0,1)$, because the logarithmic function in $B$ is analytic for $t < 1$ 
		and other functions are polynomials in $t$. The Taylor series 
		of $P$ at $t = \omega$ can be written as 
		\begin{equation*}
		P(t) = \sum_{n=0}^\infty \frac{P^{(n)}(\omega)}{n!}(t-\omega)^n.
		\end{equation*}
		We have $B(\omega) = 0$, and
\begin{equation}
\label{der-B}
B'(t) = -\frac{\omega - t}{1-t}, \quad B''(t) = \frac{1-\omega}{(1-t)^2},
\end{equation}	
which implies 
\begin{equation}
\label{B-exp}
B(t) = \frac{(t-\omega)^2}{2(1-\omega)} + \frac{(t-\omega)^3}{3 (1-\omega)^2} + \frac{(t-\omega)^4}{4 (1-\omega)^3} + \mathcal{O}((t-\omega)^5).
\end{equation}
Furthermore, we define
\begin{align}
G(t) := 3A(t)^2 +  2  (1-\omega) t (3 + t) (\omega - t) \label{G-def}
\end{align}
and expand
\begin{align}
G(t) &= 3 \omega^2 - 12 \omega t + (21 - 4 \omega + \omega^2) t^2 - (20 - 8 \omega) t^3 + 3 t^4  \notag \\
&= 12 \omega^2 (1 -  \omega)^2 + 2\omega (1 - \omega)(15 - 19\omega) (t-\omega) \notag \\
& \quad + 
(1-\omega) (21 - 43 \omega) (t-\omega)^2 + \mathcal{O}((t-\omega)^3). \label{G-exp}
\end{align}	
Similarly, we expand
\begin{align}
3t A(t) &= 3 t (\omega + (\omega-3) t + t^2) \notag \\
&= -6 \omega^2 (1 -  \omega) - 15 \omega (1 - \omega) (t-\omega) - 
(9 - 12 \omega) (t-\omega)^2 + \mathcal{O}((t-\omega)^3).
\label{A-exp}
\end{align}	
Substituting (\ref{B-exp}) into (\ref{N}) yields 
\begin{align*}
P(t) &= (t-\omega)^2 \left[ 3 t A(t) + G(t) \left( \frac{1}{2(1-\omega)} + \frac{(t-\omega)}{3 (1-\omega)^2} + \frac{(t-\omega)^2}{4 (1-\omega)^3} + \mathcal{O}((t-\omega)^3) \right) \right].
\end{align*}	
By using (\ref{G-exp}) and (\ref{A-exp}), we compute coefficients of powers $(t-\omega)$ in $P(t)$:
\begin{align*}
(t-\omega)^2 : \quad & -6 \omega^2 (1 -  \omega) + \frac{12 \omega^2 (1 -  \omega)^2}{2(1-\omega)} = 0, \\
(t-\omega)^3 : \quad & - 15 \omega (1 - \omega) + \frac{2\omega (1 - \omega)(15 - 19\omega)}{2(1-\omega)} + \frac{12 \omega^2 (1 -  \omega)^2}{3 (1-\omega)^2} = 0, \\
(t-\omega)^4 : \quad & - (9 - 12 \omega) + \frac{(1-\omega) (21 - 43 \omega)}{2(1-\omega)} + \frac{2\omega (1 - \omega)(15 - 19\omega)}{3 (1-\omega)^2} + \frac{12 \omega^2 (1 -  \omega)^2}{4 (1-\omega)^3} \\
& = \frac{9-6\omega-\omega^2}{6(1-\omega)},
\end{align*}
This yields (\ref{Taylor-coeff}).	

The remainder of $P(t)$ can be written in the integral form:
\begin{equation*}
P (t)=    \frac{1}{3!} (t-\omega)^4 \int_0^1 (1-s)^3P^{(4)}(\omega+s(t-\omega))\,ds. 
\end{equation*}
and there exists $\delta' >0$, such that $P^{(4)}(t)$ is continuous on $t\in (\omega-\delta',\omega + \delta')$. By taking $\epsilon = \frac{1}{2}P^{(4)}(\omega)$ and $\delta<\delta'$, there is a local strictly positive estimation 
\begin{equation*}
P (t) \geq  \frac{1}{48}P^{(4)}(\omega) (t-\omega)^4 >0, \quad \quad t\in [\omega-\delta,\omega ) \cup ( \omega, \omega + \delta], 
\end{equation*}
which yields the assertion on positivity of $P(t)$ near $t = \omega$. 
\end{proof}

To estimate the global behavior of the function $P(t)$ for $t \in (0,1)$, 
we use the following bounds on the function $B(t)$ obtained from (\ref{der-B}).

\begin{lemma}
	\label{lemma2}
	The function $B$ defined in (\ref{AB}) can be estimated as 
	\begin{equation}\label{B1}
	\frac{(\omega - t)^2}{2(1-t)} \leq B(t) \leq \frac{(\omega - t)^2}{2(1-\omega)}, \quad \quad  t\in (0,\omega), 
	\end{equation}
	and
	\begin{equation}\label{B2}
	B(t) \leq \frac{(\omega - t)^2}{2(1-t)}, \quad \quad t\in (\omega, 1). 
	\end{equation} 
\end{lemma}
	
	\begin{proof}
It follows from (\ref{der-B}) that $B$ can be written in the integral form:
		\begin{equation*}
		B(t) = \int_t^\omega \frac{\omega-s}{1-s}\,ds, \quad t \in (0,1).
		\end{equation*}
		For $t\in (0,\omega)$, let $0<t\leq s \leq \omega$, so that $ \frac{1}{1-t} \leq \frac{1}{1-s} \leq \frac{1}{1-\omega}$. Then, we have
		\begin{equation*}
		B(t) \leq \frac{1}{1-\omega} \int_t^\omega (\omega - s)\,ds = \frac{(\omega - t)^2}{2(1-\omega)}
		\end{equation*}
		and 
		\begin{equation*}
		B(t) \geq \frac{1}{1-t} \int_t^\omega (\omega - s)\,ds = \frac{(\omega-t)^2}{2(1-t)}.
		\end{equation*}
		This yields (\ref{B1}). Similarly, for $t\in (\omega,1)$, let $\omega \leq s \leq t < 1$, so that $\frac{1}{1-t} \geq \frac{1}{1-s}$. Then, we have
		\begin{equation*}
		B(t) \leq \frac{1}{1-t} \int_t^\omega (\omega - s)\,ds = \frac{(\omega-t)^2}{2(1-t)}.
		\end{equation*}
This yields (\ref{B2}).
	\end{proof}

We use Lemma \ref{lemma2} to extend Lemma \ref{lemma1} and to guarantee that $P(t)$ is positive for every $t \in (0,1)$. This is obtained by controlling the derivative of $P$ in $t$ separately for $t \in (0,\omega)$ and $t \in (\omega,1)$.

\begin{lemma}\label{lemma3}
	The function $P$ is monotonically decreasing on $(0,\omega)$ and  increasing on $(\omega,1)$. 
\end{lemma}
	
	\begin{proof}
To show that $P'(t) < 0$ for $t \in (0,\omega)$ and $P'(t) > 0$ for $t \in (\omega, 1)$, we use (\ref{N}) rewritten as 
		\begin{equation*}
		P(t) = Q(t) + B(t) G(t), 
		\end{equation*}
where $Q(t) := 3 t (\omega - t)^2 A(t)$ with $A$ and $B$ defined in (\ref{AB}) and $G$ defined in (\ref{G-def}). By using (\ref{der-B}) for $B'(t)$, as well as  
\begin{align*}
Q'(t) &=3(t-\omega) \left[ 5t^3+(\omega-12)t^2+ \omega (9-2\omega) t-\omega^2 \right], \\
G'(t) &= 2 \left[ 6 t^3 + 6 (2\omega - 5) t^2 + (\omega^2 - 4 \omega + 21) t - 6\omega \right],  
\end{align*}
we obtain 	
\begin{align}
P'(t) &= Q'(t) + B'(t)G(t) + B(t)G'(t)  \notag \\
&= \left[ B(t) - \frac{(\omega-t)^2}{2(1-t)} \right] G'(t) + \frac{(\omega-t)^2}{1-t} \left[ \frac{1}{2} G'(t) + \frac{1-t}{(\omega - t)^2} (Q'(t) + B'(t) G(t)) \right] \notag \\
&= \left[ B(t) - \frac{(\omega-t)^2}{2(1-t)} \right ]G'(t) -\frac{(\omega-t)^3}{1-t} \left[ 6 + (1-\omega) t - 6t^2 \right], \label{Pprime}
\end{align}
where the last identity is derived directly from 
\begin{align*}
& \frac{1-t}{(\omega - t)^2} (Q'(t) + B'(t) G(t)) \\
&= 
\frac{1}{t - \omega} \left[ 3 (1-t) \left[ 5t^3+(\omega-12)t^2+ \omega (9-2\omega) t-\omega^2 \right] + G(t) \right]  \\
&= \frac{1}{t - \omega} \left[ -12 t^4 + (31 + 5 \omega) t^3 + (7 \omega^2 - 28 \omega - 15) t^2 + 3 \omega (5 - \omega) t \right] \\
&= -12 t^3 + (31 - 7 \omega) t^2 + 3 (\omega - 5) t 
\end{align*}
and
\begin{align*}
\frac{1}{2} G'(t) + \frac{1-t}{(\omega - t)^2} (Q'(t) + B'(t) G(t))
&= -6 t^3 + (5 \omega + 1) t^2 + (\omega^2 - \omega + 6) t - 6\omega \\
&= (\omega - t) \left[ 6 t^2 + (\omega - 1) t - 6 \right].
\end{align*}

\underline{For $t\in (0,\omega)$,} we use the estimate (\ref{B1})  and obtain 
		\begin{equation}\label{B3}
		0 \leq  B(t) - \frac{(\omega-t)^2}{2(1-t)} \leq \frac{(\omega -t)^2}{2(1-\omega)} - \frac{(\omega-t)^2}{2(1-t)} = \frac{(\omega-t)^3}{2(1-t)(1-\omega)}.  
		\end{equation}
Since 
$$
6 + (1-\omega) t - 6t^2 \geq \min\{ 6, 1-\omega\} > 0, \quad  t \in [0,1],
$$ 
for every $\omega \in (0,1)$, it follows from (\ref{Pprime}) and the lower bound in (\ref{B3}) that $P'(t) < 0$ for $t \in (0,\omega)$ if $G'(t) < 0$. 
On the other hand, if $G'(t) > 0$, then we use the upper bound in (\ref{B3}) and obtain 
\begin{align*}
P'(t) &\leq \frac{(\omega-t)^3}{(1-t)(1-\omega)} \left[ \frac{1}{2} G'(t) - (1-\omega)(6+t-\omega t - 6t^2)  \right] \\
&= \frac{2 (\omega-t)^3}{(1-t)(1-\omega)} \left[ 3 t^3 + 3 (\omega - 4) t^2 + (10-\omega) t - 3 \right].
\end{align*}
We show that the last expression in the brackets is negative, which yields 
$P'(t) < 0$ for $t \in (0,\omega)$ if $G'(t) > 0$. Indeed, we have 
\begin{align*}
3 t^3 + 3 (\omega - 4) t^2 + (10-\omega) t - 3 = 3(t-1)^3 - (1-\omega) t (3t - 1)
\end{align*}
which implies 
\begin{align*}
3(t-1)^3 - (1-\omega) t (3t - 1) \leq 3 (t-1)^3 < 0, \quad  \frac{1}{3} \leq t < 1
\end{align*}
and 
\begin{align*}
3(t-1)^3 - (1-\omega) t (3t - 1) = (t-1) \left[ 3 (t-1)^2 + \frac{1-\omega}{1-t} t (3t - 1) \right] < 0, \quad  0 < t \leq \frac{1}{3}, 
\end{align*}
since $\frac{1-\omega}{1-t} < 1$ for $t \in (0,\omega)$ and 
\begin{align*}
 3 (t-1)^2 + \frac{1-\omega}{1-t} t (3t - 1) \geq 
  3 (t-1)^2 + t (3t - 1) = 3 - 7t + 6 t^2 \geq \frac{4}{3}, \qquad 0 \leq t \leq \frac{1}{3}.
\end{align*}

\underline{For $t\in (\omega,1)$,} we use again that 
		\begin{equation*}
\frac{1}{2} G'(t) - (1-\omega)(6+ (1-\omega) t - 6t^2) = 2 \left [ 3(t-1)^3 - (1-\omega ) t (3t- 1) \right ] < 0, 
		\end{equation*}
which yields 
		\begin{equation*}
		\frac{G'(t)}{6+ (1-\omega) t - 6t^2} < 2(1-\omega). 
		\end{equation*}
By using (\ref{B2}), we know that $B(t) -  \frac{(\omega - t)^2}{2(1-t)}\leq 0$, 
so that we can estimate (\ref{Pprime}) for $t \in (\omega,1)$ as follows: 
		\begin{align*}
		P'(t) &= (6 + (1-\omega) t  - 6t^2) \left \{ \left [ B(t) - \frac{(\omega-t)^2}{2(1-t)} \right ] \frac{G'(t)}{6+ (1-\omega) t - 6t^2}  -\frac{(\omega-t)^3}{1-t} \right \} \\
		& > (6 + (1-\omega) t  - 6t^2) \left \{  2(1-\omega) \left [ B(t) - \frac{(\omega-t)^2}{2(1-t)} \right ] -\frac{(\omega-t)^3}{1-t} \right \} \\
				& =  (6 + (1-\omega) t  - 6t^2) \left \{  2 (1-\omega) B(t) - (t-\omega)^2 \right \}. 
		\end{align*}
By using the definition of $B$ in (\ref{AB}) and the variable $x := \frac{t- \omega}{1-\omega} \in (0,1)$, we get Taylor series expansion 
		\begin{align*}
2 (1-\omega) B(t) - (t-\omega)^2 &= 2 (1-\omega)^2 \log\frac{1-\omega}{1-t} - 2(1-\omega) (t-\omega) - (t-\omega)^2 \\
& =  2 (1-\omega)^2 \left[ -\log(1-x) - x -\frac{1}{2} x^2 \right] \\
&= 2 (1-\omega)^2 \sum_{n=3}^{\infty} \frac{x^n}{n},
\end{align*}
which is strictly positive for $x \in (0,1)$. Hence, $P'(t) > 0$ for $t\in (\omega, 1)$. 
\end{proof}

The period function $T(\mathcal{E},\omega)$ given by (\ref{period-function}) 
can be rewritten for the even periodic waves explicitly by
\begin{equation}
\label{even-period}
T_\mathrm{even}(\mathcal E,\omega) = 2 \int_{m}^{M}\frac{d \phi }{\sqrt{2\mathcal E + (1-\omega)\log(1-\phi^2) - (1-\omega) \log(1-\omega) + \phi^2 - \omega}},
\end{equation}
where 
$$
m := \min\limits_{x\in [-\frac{L}{2},\frac{L}{2}]} \phi (x) \in (0,\sqrt{\omega}) \quad \mbox{\rm and} \quad 
M:= \max\limits_{x\in [-\frac{L}{2},\frac{L}{2}]} \phi(x)\in (\sqrt{\omega},1)
$$ 
are given by roots of 
$V(\phi) = \mathcal{E}$ for $\mathcal{E}\in (0,\mathcal{E}_{\omega})$, see Remark \ref{boundphi}. By using Lemma \ref{lemma3}, we prove monotonicity of the period function in $\mathcal{E}$ stated in Theorem \ref{th-period}.

\begin{proposition}
	\label{prop-even-period}
	For every $\omega \in (0,1)$, the period function $T_\mathrm{even}(\mathcal E,\omega)$ given by (\ref{even-period}) is monotonically increasing 
	in $\mathcal{E} \in (0,\mathcal{E}_{\omega})$ such that 
	$$
\lim_{\mathcal{E}\to 0} 	T_\mathrm{even}(\mathcal E,\omega) = 2\pi \sqrt{\frac{1-\omega}{2\omega}}, \quad 
\lim_{\mathcal{E}\to \mathcal{E}_{\omega}} T_\mathrm{even}(\mathcal E,\omega) = +\infty,
	$$  
\end{proposition}

	\begin{proof}
		Lemma \ref{lemma3} implies that $P(t) > 0$ for $t \in (0,1) \backslash \{\omega\}$, which yields $I''(\phi) > 0$ for $\phi \in (0,1)$ by (\ref{I-second}). Since $V(\phi) \geq 0$ and $V(\sqrt{\omega}) = 0$, we can apply the main theorem from \cite{chicone2} by using the translated coordinate $\varphi = \phi - \sqrt{\omega}$. Since $I''(\phi) > 0$, the main theorem of \cite{chicone2} states that the period function $T_\mathrm{even}(\mathcal E,\omega)$ is monotonically increasing in $\mathcal{E} \in (0,\mathcal{E}_{\omega})$ for every $\omega \in (0,1)$. The limit for $T_{\mathrm{even}}(\mathcal E,\omega)$ as 
		$\mathcal{E} \to 0$ follows from the linearization of the center 
		point $(\sqrt{\omega},0)$.
		The divergence of $T_\mathrm{even}(\mathcal E,\omega)$ as $\mathcal{E} \to \mathcal{E}_{\omega}$ follows from the infinite period of the homoclinic orbit to the saddle equilibrium point $(0,0)$. 
	\end{proof} 

\subsection{Monotonicity for odd periodic waves}
\label{sec-3-2}

The period function $T(\mathcal{E},\omega)$ given by (\ref{period-function}) 
can be rewritten for the odd periodic waves explicitly by
\begin{equation}
\label{odd-period}
T_\mathrm{odd}(\mathcal E, \omega) = 4 \int_0^{M} \frac{d \phi }{\sqrt{2\mathcal E + (1-\omega)\log(1-\phi^2) - (1-\omega) \log(1-\omega) + \phi^2 - \omega}},
\end{equation}
where 
$$
M := -\min\limits_{x\in [-\frac{L}{2},\frac{L}{2}]} \phi (x) \in (0,\sqrt{\omega}) = \max\limits_{x\in [-\frac{L}{2},\frac{L}{2}]} \phi(x)\in (\sqrt{\omega},1)
$$ 
is a positive root of 
$V(\phi) = \mathcal{E}$ for $\mathcal{E}\in (\mathcal{E}_{\omega},\infty)$, see Remark \ref{remodd}. 
The proof of monotonicity of the period function in $\mathcal{E}$ is easier for the odd periodic waves. The following proposition justifies the result stated in Theorem \ref{th-period}.

\begin{proposition}
	\label{prop-odd-period}
		For every $\omega \in (0,1)$, the period function $T_\mathrm{odd}(\mathcal E,\omega)$ given by (\ref{odd-period}) is monotonically decreasing in $\mathcal{E} \in (\mathcal{E}_{\omega},\infty)$ such that 
		$$
\lim_{\mathcal{E}\to \mathcal{E}_{\omega}} T_\mathrm{odd}(\mathcal E,\omega) = +\infty,  \quad \lim_{\mathcal{E} \to \infty} T_\mathrm{odd}(\mathcal E,\omega) = 0.
		$$
\end{proposition}

\begin{proof}
Using the same transformation $t = \phi^2$ as in Section \ref{sec-3-1}, we
redefine $V(\phi)$ in (\ref{energy-again}) as 
		\begin{equation*}
		W(t) := \frac{1}{2} (\omega - t) + \frac{1}{2}(1 - \omega) 
		\log \frac{1-\omega}{1-t}, \quad t\in (0,1).
		\end{equation*} 
Similarly, we redefine $M \in (\sqrt{\omega},1)$ as $q := M^2 \in (\omega, 1)$. Since $\mathcal E = V(M) = W(q)$, we use the change of variables $t = \phi^2$ for $t \in (0,q)$ and $t = q u$ for $u \in (0,1)$ and rewrite the integral (\ref{odd-period}) in the equivalent form:
		\begin{align*}
T_\mathrm{odd}(\mathcal E, \omega) &= 4 \int_0^{M} \frac{d\phi}{\sqrt{2[V(M) - V(\phi) ]}} \\
&= \sqrt{2}\int_0^{q} \frac{dt}{\sqrt{t[W(q) - W(t)]}} \\
		&=  \int_0^{1} \frac{\sqrt{2q}}{\sqrt{u[W(q) - W(qu)]}}\,du.
		\end{align*} 
Since $V'(\phi) > 0$ for $\phi \in (\sqrt{\omega},1)$, we have $W'(q) > 0$. The chain rule
\begin{equation*}
\frac{\partial T_{\rm odd}}{\partial \mathcal E} = \frac{\partial T_{\rm odd}}{\partial q} \left ( \frac{\partial \mathcal E }{\partial q} \right )^{-1} = \frac{1}{W'(q)} \frac{\partial T_{\rm odd}}{\partial q}
\end{equation*} 	
implies that for a fixed $\omega \in (0,1)$, monotonicity of $T_{\rm odd}(\mathcal{E},\omega)$ in $\mathcal{E}$ and $q$ coincide. 
Although the integral for $T_{\rm odd}(\mathcal{E},\omega)$ is weakly singular at $u = 0$ and $u = 1$, the derivative of $T_{\rm odd}(\mathcal{E},\omega)$ in $q$ yields also weakly singular integrals and, hence, it can be computed by pointwise differentiation as in 
		\begin{equation*}
		\frac{\partial T_{\rm odd}}{\partial q} = \frac{1}{\sqrt{2q}}  \int_0^{1} \frac{du}{\sqrt{u[W(q) - W(qu)]}} - \frac{\sqrt{2q}}{2}  \int_0^{1} \frac{W'(q) - uW'(qu)}{\sqrt{u [W(q) - W(qu)]^3}}\,du, 
		\end{equation*} 
where the second integral remains weakly singular at $u = 1$ since 
$W(q) - W(uq) = \mathcal{O}(1-u)$ and $W'(q) - uW'(qu) = \mathcal O(1-u)$ as $u\to 1$. The function $W(t)$ is strictly convex since
		\begin{equation*}
		W'(t) = \frac{t-\omega}{2(1-t)} \quad \quad W''(t) = \frac{1-\omega}{2(1-t)^2} > 0. 
		\end{equation*} 
If $F(t):=tW'(t)-W(t)$, then $F'(t) = tW''(t) > 0$ for $t \in (0,1)$, so that $F(q) > F(qu)$ for every $u\in (0,1)$. This implies for $u \in (0,1)$ that 
		\begin{equation*} 
qW'(q)-W(q) >  quW'(qu)-W(qu), \quad \Rightarrow \quad 
		W(q) - W(qu) < q [ W'(q) - u W'(qu)].
		\end{equation*} 
		Since $\sqrt{u [W(q) - W(qu)]} > 0$ for $u \in (0,1)$, it follows that 
		\begin{equation*}
		\frac{1}{\sqrt{q}} \frac{1}{\sqrt{u[W(q) - W(qu)]}} < \sqrt{q} \frac{W'(q) - uW'(qu)}{\sqrt{u}\left [W(q) - W(qu) \right ]^{3/2}}, \quad u \in (0,1),
		\end{equation*} 
which proves that 
$$
		\frac{\partial T_{\rm odd}}{\partial q} < 0, \quad q \in (\omega,1),
$$
This yields the desired monotonicity in $\mathcal{E}$ by the chain rule. 		The divergence of $T_\mathrm{odd}(\mathcal E,\omega)$ as $\mathcal{E} \to \mathcal{E}_{\omega}$ follows from the infinite period of the homoclinic orbit to the saddle equilibrium point $(0,0)$.
The zero limit of $T_\mathrm{odd}(\mathcal E,\omega)$ as $\mathcal{E} \to \infty$  follows from (\ref{odd-period}) by the dominated convergence theorem since $M \in (\sqrt{\omega},1)$ is finite.
	\end{proof}

\begin{remark}
	The result of Proposition \ref{prop-odd-period} is true for $\omega \in (-\infty,0)$ with the only change 
	$$
	\lim_{\mathcal{E} \to \mathcal{E}_{\omega}} T_{\rm odd}(\mathcal{E},\omega) = \frac{2\pi}{\sqrt{|\omega|}}. 
	$$
	which is computed from the linearization of the center point $(0,0)$ for $\omega \in (-\infty,0)$. All other computations are identical to the proof of Proposition \ref{prop-odd-period}. 
\end{remark}
			
\section{Spectral analysis near the periodic waves}
\label{sec-4}
		
Consider the Hessian operator $\mathcal{L} = H''(\phi) + \omega Q''(\phi)$ defined in $(\ref{matrixop})$ as an operator on $\mathbb{L}^2_{\rm per}$ with the domain in $\mathbb{H}^2_{\rm per}$. Since $\mathcal{L}$ is a diagonal composition 
of the Schr\"{o}dinger operators $\mathcal{L}_{+}$ and $\mathcal{L}_{-}$ in $L^2_{\rm per}$ with the domain in $H^2_{\rm per}$, 
the spectrum of $\mathcal{L}$ is a superposition of the spectra 
of $\mathcal{L}_+$ and $\mathcal{L}_-$. According to \cite{Magnus}, 
the spectrum of either $\mathcal{L}_+$ or $\mathcal{L}_-$ consists
of an unbounded sequence of real eigenvalues
		\[
		\lambda_0 < \lambda_1 \leq \lambda_2 < \lambda_3 \leq \lambda_4\;\; ...\; < \lambda_{2n-1} \leq
		\lambda_{2n}\; \cdots, 
		\]
where equality means that $\lambda_{2n-1} = \lambda_{2n}$  is a
double eigenvalue. By \cite[Theorem 3.1.2(ii)]{eastham}, 
if $\varphi$ is an eigenfunction associated to the eigenvalue $\lambda_{2n-1}$ or
$\lambda_{2n}$, then $\varphi$ has exactly $2n$ zeroes on the periodic domain. 

To characterize the Morse index of $\mathcal{L}_{\pm}$ denoted by $n(\mathcal{L}_{\pm})$ and the nullity index of $\mathcal{L}_{\pm}$ denoted by $z(\mathcal{L}_{\pm})$, we use the following theorem, see \cite[Theorem 3.1]{neves}.
		
\begin{theorem}
	\label{teo12}
	Let $\mathcal{M} = -\partial_x^2+Q(x)$ be a linear Schr\"{o}dinger operator with the even, $L$-periodic, bounded potential $Q$ and let $\{ \varphi_1, \varphi_2\}$ be linearly independent solutions of $\mathcal{M} \varphi = 0$ satisfying 
\begin{equation}
\label{ivp}
	\left\{ \begin{array}{l} \varphi_1(0) = 1, \\
	\varphi_1'(0) = 0, \end{array} \right. \quad \mbox{\rm and} \quad 
	\left\{ \begin{array}{l} \varphi_2(0) = 0, \\
\varphi_2'(0) = 1. \end{array} \right.	
\end{equation}
Assume that there exists $\theta \in \mathbb{R}$ such that 
\begin{equation}
\label{theta}
	\varphi_1(x+L) = \varphi_1(x) + \theta \varphi_2(x), \quad 
\mbox{\rm and} \quad  \varphi_2(x+L) = \varphi_2(x), 
\end{equation}
and that the $L$-periodic eigenfunction $\varphi_2$ has two zeros on the periodic domain. The zero eigenvalue of $\mathcal{M}$ in $L^2_{\rm per}$ is simple if $\theta \neq 0$ and double if $\theta = 0$. It is the second eigenvalue of $\mathcal{M}$ if $\theta \geq 0$ and the third eigenvalue of $\mathcal{M}$ if $\theta < 0$. 
\end{theorem}	

\begin{remark}
Since the linear operator $\mathcal{L}_+$ is related to the linearization 
of the second-order equation (\ref{odephi}) on the periodic orbit with the profile $\phi$, the two solutions in Theorem \ref{teo12} are constructed from the first invariant (\ref{energy}) and the 
parameter $\theta$ can be computed from the derivative of the period 
function $T(\mathcal{E},\omega)$ with respect to $\mathcal{E}$. See \cite[Section 3.2]{GeyerPel25}. 	
\end{remark}
		
\subsection{Spectral analysis of even periodic waves}
		
We proceed separately with the analysis of the Schr\"{o}dinger operators $\mathcal{L}_+$ and $\mathcal{L}_-$ defined in $(\ref{matrixop})$ and computed at the even waves of Theorem \ref{th-existence} with the profile $\phi$ satisfying (\ref{even-wave}).

\begin{proposition}\label{propR0}
$n(\mathcal{L}_+) = z(\mathcal{L}_+) = 1$, that is, 
$0$ is a simple eigenvalue of $\mathcal{L}_{+}$ associated with the eigenfunction $\phi'$, and there is only one negative eigenvalue, which is simple. In addition, the remainder of the spectrum of $\mathcal{L}_+$ in $L^2_{\rm per}$ consists of a discrete set of positive eigenvalues with finite multiplicities.
\end{proposition}

\begin{proof} 
	On comparison with $\mathcal{M}$ in Theorem \ref{teo12}, we have 
\begin{equation}
\label{Q-def}
	Q = 1 + (\omega - 1) \frac{1+ \phi^2}{(1-\phi^2)^2},
\end{equation}
	where $0 < \phi < 1$ is the spatial profile of the $L$-periodic orbit in Theorem \ref{th-existence} satisfying (\ref{even-wave}) with $x_0 = 0$ and $\omega \in (\omega_L,1)$. Hence, $Q$ is even, $L$-periodic, and bounded. 
	
	Consider the family of periodic orbits of the second-order equation (\ref{odephi}) associated with the period function $T(\mathcal{E},\omega)$ for the energy level $\mathcal{E} = E(\phi,\phi')$ given by the first invariant (\ref{energy}) with $\mathcal{E}\in (0,\mathcal{E}_{\omega})$. Due to monotonicity 
	of the mapping $\mathcal{E}\to T(\mathcal{E},\omega)$ for fixed $\omega \in (0,1)$ in Theorem \ref{th-period}, there exists a unique $\mathcal{E} = \mathcal{E}_L(\omega)$ of $T(\mathcal{E}_L(\omega),\omega) = L$ for a fixed spatial period $L > 0$ and $\omega \in (\omega_L,1)$. We further define $\phi_L(\omega) \in (0,1)$ as a root of $V(\phi) = \mathcal{E}$ for $\mathcal{E} = \mathcal{E}_L(\omega)$. Two roots exist for the maximum and minimum of the spatial profile $\phi$, see Remark \ref{boundphi}. Since 
	$$
	V'(\phi) = -\frac{\phi(\omega - \phi^2)}{1-\phi^2},
	$$ 
	we have $V'(\phi_L(\omega)) \neq 0$ for either choice for $\phi_L(\omega)$. 
	Equations (\ref{odephi}) and (\ref{energy}) imply that 
\begin{equation}
\label{property-energy}
	\phi''(0) = -V'(\phi_L(\omega)) \quad \mbox{\rm and} \quad 
	\frac{\partial \phi(0)}{\partial \mathcal{E}} \biggr|_{\mathcal{E} = \mathcal{E}_L(\omega)} = \frac{1}{V'(\phi_L(\omega))},
\end{equation}
where the family of periodic orbits parameterized by $\mathcal{E}$ is restricted to even functions by using 
the translational invariance of the second-order equation (\ref{odephi}).
	
Since $\mathcal{L}_+$ is a linearized operator for (\ref{odephi}), we obtain two linearly independent solutions of $\mathcal{L}_+ \varphi = 0$ in Theorem \ref{teo12} by using 
\begin{equation}
\label{two-solutions}
	\varphi_1(x) = \frac{\partial \phi(x)}{\partial \mathcal{E}} \biggr|_{\mathcal{E} = \mathcal{E}_L(\omega)} V'(\phi_L(\omega)), \quad \varphi_2(x) = -\frac{\phi'(x)}{V'(\phi_L(\omega))}.
\end{equation}
Since $\phi$ is even, we obtain (\ref{ivp}) from (\ref{property-energy}).
The second solution $\varphi_2$ is $L$-periodic and has two zeros on the periodic domain according to the assumption of Theorem \ref{teo12}. Computing the first solution $\varphi_1$ after the period $L$, we obtain 
$$
	\varphi_1(L) =   \frac{\partial \phi(L)}{\partial \mathcal{E}} \biggr|_{\mathcal{E} = \mathcal{E}_L(\omega)} V'(\phi_L(\omega)) \quad \mbox{\rm and} \quad \varphi_1'(L) =  \frac{\partial \phi'(L)}{\partial \mathcal{E}} \biggr|_{\mathcal{E} = \mathcal{E}_L(\omega)} V'(\phi_L(\omega))
	=: \theta.
$$
Since $\phi(T(\mathcal{E},\omega)) = \phi(0)$ and 
$\phi'(T(\mathcal{E},\omega)) = 0$, taking derivative of these equations in $\mathcal{E}$ at the energy level $\mathcal{E} = \mathcal{E}_L(\omega)$ 
implies that $\varphi_1(L) = 1$ and 
$$
\theta = - \frac{\partial T}{\partial \mathcal{E}} \biggr|_{\mathcal{E} = \mathcal{E}_L(\omega)} \phi''(0) V'(\phi_L(\omega)) = 
\frac{\partial T}{\partial \mathcal{E}} \biggr|_{\mathcal{E} = \mathcal{E}_L(\omega)} \left[ V'(\phi_L(\omega)) \right]^2, 
$$
where we have used (\ref{property-energy}) again. The $L$-periodicity of $Q$ implies that $\varphi_1$ satisfies (\ref{theta}) with the sign of $\theta$ given by the sign of the derivative of the mapping $\mathcal{E} \to T(\mathcal{E},\omega)$ at $\mathcal{E} = \mathcal{E}_L(\omega)$. Since $\theta > 0$ by Proposition \ref{prop-even-period}, Theorem \ref{teo12} proves the assertion. 	
\end{proof}

\begin{remark}
	\label{smoothcurve1}		 	
Let $L>0$ be fixed. Using the implicit function theorem and the fact that $\Ker(\mathcal{L}_{+})={\rm Span}(\phi')$ with $\phi'$ being odd, it is possible to prove that the mapping $\omega \mapsto \phi \in H_{\rm per,e}^2$ is $C^1$ for every $\omega \in (\omega_L,1)$. In addition, differentiating  (\ref{odephi}) with respect to $\omega$ yields the derivative equation:
\begin{equation}
\label{der-omega}
\mathcal{L}_{+} \frac{d\phi}{d\omega} = - \frac{\phi}{1-\phi^2}.
\end{equation}
This improves Proposition \ref{prop2}, where the mapping 
$\omega \mapsto \phi \in H_{\rm per,e}^2$ is only stated to be continuous 
for every $\omega \in (\omega_L,1)$.
\end{remark}
		 
\begin{proposition}
	\label{propL2} 
$n(\mathcal{L}_-) = 0$ and $z(\mathcal{L}_-) = 1$, that is, 
	$0$ is a simple eigenvalue of $\mathcal{L}_{-}$ associated with the eigenfunction $\phi$ and the remainder of the spectrum of $\mathcal{L}_-$ in $L^2_{\rm per}$ consists of a discrete set of positive eigenvalues with finite multiplicities.
\end{proposition}

\begin{proof} 
	Since $0<\phi<1$ we obtain from the definition of $\mathcal{L}_{-}$ that 
	$$
\mathcal{L}_{-}=- \partial_x^2+ \frac{\omega-\phi^2}{1-\phi^2}.
	$$ 
	Since $\phi$ is positive and satisfies $(\ref{ode})$, we obtain that $\mathcal{L}_{-}\phi=0$. By standard Floquet theory in \cite{Magnus}, we deduce that zero is the first eigenvalue of $\mathcal{L}_{-}$ which is simple.  Again, the last part of the proposition is obtained from the fact that $\mathcal{L}_-$ is a self-adjoint operator and the compact embedding $H_{per}^2\hookrightarrow L_{per}^2.$
\end{proof}

Propositions $\ref{propR0}$ and $\ref{propL2}$ imply the following result for the case of even periodic waves.

\begin{corollary}
	\label{propL} 
The Hessian operator $\mathcal{L}$ defined by (\ref{matrixop}) in  $\mathbb{L}^2_{\rm per}$ with domain $\mathbb{H}^2_{\rm per}$ has one negative eigenvalue which is simple. Zero is a double eigenvalue with associated eigenfunctions $(\phi',0)$ and $(0,\phi)$.  In addition, the remainder of the spectrum consists of a discrete set of positive eigenvalues with finite multiplicities.
\end{corollary}

\subsection{Spectral analysis of odd periodic waves}

We proceed separately with the analysis of the Schr\"{o}dinger operators $\mathcal{L}_+$ and $\mathcal{L}_-$ defined in $(\ref{matrixop})$ and computed at the odd waves of Theorem \ref{th-existence} with the profile $\phi$ satisfying (\ref{odd-wave}). 
		 
\begin{proposition}\label{propL11} 
$n(\mathcal{L}_+) = 2$ and $z(\mathcal{L}_+) = 1$, that is, 
	$0$ is a simple eigenvalue of $\mathcal{L}_{+}$ associated with the eigenfunction $\phi'$, and there are two negative simple eigenvalues. The remainder of the spectrum of $\mathcal{L}_+$ in $L^2_{\rm per}$ consists of a discrete set of positive eigenvalues with finite multiplicities.
\end{proposition}

\begin{proof} 
We can prove the assertion in two different ways.

\underline{Proof I.} The potential $Q$ in the linear operator $\mathcal{M}$ of Theorem \ref{teo12} is defined by the same expression (\ref{Q-def}), 
where $-1 < \phi < 1$ is the spatial profile of the $L$-periodic orbit in Theorem \ref{th-existence} satisfying (\ref{odd-wave}) with $x_0 = 0$ and $\omega \in (\Omega_L,1)$. Hence, $Q$ is even, $L$-periodic, and bounded. 
Since $\phi$ is even with respect to $x = \frac{L}{4}$ due to the second property in (\ref{odd-wave}), $Q$ has the minimum period $\frac{L}{2}$ and it is also even with respect to $x = \frac{L}{4}$. Therefore, we can repeat the proof of Proposition \ref{propR0} and introduce the family of odd periodic orbits for the energy level $\mathcal{E} = E(\phi,\phi')$ with $\mathcal{E}\in (\mathcal{E}_{\omega},\infty)$. Again, due to monotonicity of the mapping $\mathcal{E}\to T(\mathcal{E},\omega)$ for fixed $\omega \in (-\infty,1)$ in Theorem \ref{th-period}, there exists a unique $\mathcal{E} = \mathcal{E}_L(\omega)$ of $T(\mathcal{E}_L(\omega),\omega) = L$ for a fixed spatial period $L > 0$ and $\omega \in (\Omega_L,1)$. We further define $\phi_L(\omega) \in (0,1)$ as a unique root of $V(\phi) = \mathcal{E}$ for $\mathcal{E} = \mathcal{E}_L(\omega)$, see Remark \ref{remodd}, with the same 
property (\ref{property-energy}) and the same definition (\ref{two-solutions}) of 
two solutions of $\mathcal{L}_+ \varphi = 0$. 

To satisfy the initial data in (\ref{ivp}) for the two solutions, 
we can use the translational invariance of the second-order equation (\ref{odephi}) and translate the family of odd periodic orbits to the family of even periodic orbits by 
\begin{equation}
\label{translation-quarter}
\phi(x) \to \phi\left(x - \frac{1}{4} T(\mathcal{E},\omega)\right).
\end{equation}
Then, assumptions of Theorem \ref{teo12} are satisfied and 
the second solution $\varphi_2$ is $L$-periodic and has two zeros on the periodic domain, whereas the first solution $\varphi_1$ satisfies (\ref{theta}) with the same definition of $\theta$:
$$
\theta = \frac{\partial T}{\partial \mathcal{E}} \biggr|_{\mathcal{E} = \mathcal{E}_L(\omega)} \left[ V'(\phi_L(\omega)) \right]^2.
$$
Since $\theta < 0$ by Proposition \ref{prop-odd-period}, Theorem \ref{teo12} proves the assertion for every $\omega \in (\Omega_L,1)$.

\underline{Proof II.} We define the restrictions of $\mathcal{L}_{+}$ to the odd and even subspaces $L_{\rm per,o}^2\subset L_{\rm per}^2$ and $L_{\rm per,e}^2\subset L_{\rm per}^2$ and denote them by  $\mathcal{L}_{+,\rm o}$ and $\mathcal{L}_{+,\rm e}$, respectively. Since $\phi$ is odd,  $\phi'$ is an element of $\Ker(\mathcal{L}_{+,\rm e})$ but is not an element of $\Ker(\mathcal{L}_{+,\rm o})$. Using (\ref{odephi}), we have for any $\omega \in (\Omega_L,1)$,
$$
		 	(\mathcal{L}_{+,\rm o}\phi,\phi)_{L_{\rm per}^2} = 2(\omega-1)\int_0^L\frac{\phi^4}{(1-\phi^2)^2}dx<0.
$$
This implies by Courant's minimax characterization of eigenvalues of the self-adjoint operator $\mathcal{L}_{+,\rm o}$ that $n(\mathcal{L}_{+,\rm o}) \geq 1$.

By  Krein-Rutman's Theorem, the first eigenvalue of $\mathcal{L}_{+}$ is simple and it is associated to a sign-definite eigenfunction which needs to be even.  Since $0$ is an eigenvalue of $\mathcal{L}_{+,\rm e}$ associated with the 
sign-varying eigenfunction $\phi'$, this implies that $n(\mathcal{L}_{+,\rm e}) \geq 1$. Thus, we have $n(\mathcal{L}_{+}) = n(\mathcal{L}_{+,\rm o}) + n(\mathcal{L}_{+,\rm e}) \geq 2$, but since $\phi'$ has only two zeros on the periodic domain, the zero eigenvalue is nothing but the third eigenvalue of $\mathcal{L}_+$ by Theorem \ref{teo12} which further implies the assertion for every $\omega \in (\Omega_L,1)$.
\end{proof}

		  \begin{remark}
		  	\label{smoothcurve2}
Let $L>0$ be fixed. Using the implicit function theorem and the fact that $\Ker(\mathcal{L}_+)={\rm Span}(\phi')$ with $\phi'$ being even, it is possible to prove again that the mapping $\omega \mapsto \phi \in H_{\rm per,o}^2$ is $C^1$ for every $\omega \in (\Omega_L,1)$ with the same derivative equation (\ref{der-omega}). 
		 \end{remark}
		 
\begin{proposition}
	\label{propL22} 
	$n(\mathcal{L}_-) = 1$ and $z(\mathcal{L}_-) = 1$, that is, 
	$0$ is a simple eigenvalue of $\mathcal{L}_{-}$ associated with the eigenfunction $\phi$, and there is only one negative eigenvalue, which is simple. The remainder of the spectrum of $\mathcal{L}_-$ in $L^2_{\rm per}$ consists of a discrete set of positive eigenvalues with finite multiplicities.
\end{proposition}

\begin{proof}
On comparison with $\mathcal{M}$ in Theorem \ref{teo12}, we have 
	\begin{equation}
	\label{Q-def-minus}
	Q = 1 + \frac{\omega-1}{1-\phi^2},
	\end{equation}
	where $-1 < \phi < 1$ for every $\omega \in (\Omega_L,1)$. 
	Similarly to Proof I of Proposition \ref{propL11}, the $L$-periodic and bounded 
	$Q$ in (\ref{Q-def-minus}) is even with respect to both $x = 0$ and $x = \frac{L}{4}$ and has the minimal period $\frac{L}{2}$. After the translation (\ref{translation-quarter}) with $\mathcal{E}= \mathcal{E}_L(\omega)$, 
the lowest eigenvalue of $\mathcal{L}_+$ in $L^2_{\rm per,o}$ is at $0$, associated with the translated eigenfunction 
\begin{equation*}
\phi'(x) \to \phi'\left(x - \frac{L}{4} \right),
\end{equation*}	
which is now odd. It follows from the relation between $\mathcal{L}_-$ and $\mathcal{L}_+$:
$$
\mathcal{L}_- = \mathcal{L}_+ + \frac{2 (1-\omega) \phi^2}{(1-\phi^2)^2}, \quad \omega < 1, 
$$
that the lowest eigenvalue of $\mathcal{L}_-$ in $L^2_{\rm per,o}$ is greater than 
the lowest eigenvalue of $\mathcal{L}_+$ in $L^2_{\rm per,o}$. Therefore, $\mathcal{L}_-$ is strictly positive in $L^2_{\rm per,o}$. 

To study eigenvalues of $\mathcal{L}_-$ in $L^2_{\rm per,e}$ after the translation (\ref{translation-quarter}) with $\mathcal{E}= \mathcal{E}_L(\omega)$, we note that 
$\mathcal{L}_-$ has the zero eigenvalue in $L^2_{\rm per,e}$ 
associated with the translated eigenfunction 
\begin{equation*}
\phi(x) \to \phi\left(x - \frac{L}{4} \right),
\end{equation*}	
which is now even. Since this eigenfunction for the zero eigenvalue of $\mathcal{L}_-$ in $L^2_{\rm per,e}$ has two zeros on the periodic domain, 
there exists a negative eigenvalue of $\mathcal{L}_-$ in $L^2_{\rm per,e}$ 
and by Theorem \ref{teo12}, $0$ is the second simple eigenvalue of $\mathcal{L}_-$ in $L^2_{\rm per,e}$. Combining with positivity of $\mathcal{L}_-$ in $L^2_{\rm per,o}$, we have the assertion.
\end{proof}

Propositions $\ref{propL11}$ and $\ref{propL22}$ imply the following result for the case of odd periodic waves.

\begin{corollary}
	\label{propLL} 
	The Hessian operator $\mathcal{L}$ defined by (\ref{matrixop}) in  $\mathbb{L}^2_{\rm per}$ with domain $\mathbb{H}^2_{\rm per}$ has three negative eigenvalues, which are semi-simple. Zero is a double eigenvalue with associated eigenfunctions $(\phi',0)$ and $(0,\phi)$.  In addition, the remainder of the spectrum consists of a discrete set of positive eigenvalues with finite multiplicities.
\end{corollary}

\section{Constrained energy minimization of periodic waves}
\label{sec-5}

For the wave profile $\phi \in H^1_{\rm per}$ given by either even or odd periodic wave in Theorem \ref{th-existence}, we can define the energy $H(\phi)$ and mass $Q(\phi)$ computed from (\ref{Eu}) and (\ref{Fu}). We recall from Remarks  \ref{smoothcurve1} and \ref{smoothcurve2} that the mapping $\omega \to \phi \in H^1_{\rm per}$ is $C^1$ for either even or odd periodic wave. Since $\phi \in H^1_{\rm per}$ is a critical point of the augmented energy functional $G(u)$ given by (\ref{Gu}), we have
$$
\frac{d}{d\omega} G(\phi) = \frac{d}{d\omega} H(\phi) + \omega \frac{d}{d\omega} Q(\phi) + Q(\phi) = Q(\phi),
$$
which implies that the mapping $\omega \to G(\phi)$ is $C^2$ and 
$$
\frac{d^2}{d\omega^2} G(\phi) = \frac{d}{d\omega} Q(\phi) = 2 \langle \frac{\phi}{1 - \phi^2}, \frac{d \phi}{d\omega}  \rangle_{L^2_{\rm per}}.
$$
By Corollaries \ref{propL} and \ref{propLL}, the Morse index for the Hessian operator $\mathcal{L} = H''(\phi) + \omega Q''(\phi)$ given by  (\ref{matrixop}) 
is nonzero so that $\phi \in H^1_{\rm per}$ is a saddle point of $G(u)$. We further clarify if $\phi \in H^1_{\rm per}$ is a local minimizer of energy $H(u)$ under the constraint of fixed mass $Q(u)$, which is degenerate only due to symmetries. 

The NLS--IDD equation (\ref{NLS-IDD}) can be formulated as 
a Hamiltonian system in the coordinate $u = p + i q$ with $(p,q) \in \mathbb{H}^1_{\rm per}$. The two basic symmetries of the NLS--IDD equation (\ref{NLS-IDD}) are the translation and rotation symmetries. If $u = u(t,x)$ is a solution, so are $e^{-i\theta} u(t,x)$ and $u(x-\xi,t)$ for any  $\theta, \xi \in \mathbb{R}$. Considering $u = p + i q$, this yields the invariance under the two transformations given by 
\begin{equation}\label{S1}
S_1(\theta) \left(
\begin{array}{c}
p \\ q
\end{array}
\right) := \left(
\begin{array}{cc}
\cos{\theta} & -\sin{\theta} \\
\sin{\theta} & \cos{\theta}
\end{array}
\right) \left(
\begin{array}{c}
p \\ q
\end{array}
\right)
\end{equation}
and
\begin{equation}\label{S2}
S_2(\xi) \left(
\begin{array}{c}
p \\ q
\end{array}
\right) := \left(
\begin{array}{c}
p(\cdot - \xi, \cdot) \\
q(\cdot - \xi, \cdot)
\end{array}
\right).
\end{equation}
A standing wave solution of the form $u(t,x)=e^{i\omega t}\phi(x)$ is given by
\begin{equation*}
S_1(\omega t)  \left(
\begin{array}{c}
\phi(x) \\
0
\end{array}
\right) = \left(
\begin{array}{c}
\cos(\omega t) \\
\sin(\omega t)
\end{array}
\right) \phi(x).
\end{equation*}
The actions $S_1$ and $S_2$ in (\ref{S1}) and (\ref{S2}) define unitary groups in $\mathbb{H}^1_{\rm per}$ with infinitesimal generators given by
\begin{equation*}
S_1'(0) := \left(
\begin{array}{cc}
0 & -1 \\
1 & 0
\end{array}
\right) \quad \mbox{\rm and} \quad 
S_2'(0) = \left(
\begin{array}{cc}
1 & 0 \\
0 & 1
\end{array}
\right) \partial_x.
\end{equation*}
Separating the variables for the perturbation as 
$$
u(t,x)=e^{i\omega t} \left( \phi(x) + p(x,t) + i q(x,t) \right)
$$
we obtain the two-dimensional kernel of the Hessian operator 
(\ref{matrixop}) spanned by the two symmetry transformations:
$$
S_1'(0) \left(
\begin{array}{c}
\phi \\ 0
\end{array}
\right) = \left(
\begin{array}{c}
0 \\ \phi
\end{array}
\right) \quad \mbox{\rm and} \quad 
S_2'(0) \left(
\begin{array}{c}
\phi \\ 0
\end{array}
\right) = \left(
\begin{array}{c}
\phi' \\ 0
\end{array}
\right).
$$
These symmetry modes agree with the eigenfunctions in $\Ker(\mathcal{L})$ given by Corollaries \ref{propL} and \ref{propLL}. 

If we consider variation of energy $H(u)$ under fixed mass $Q(u)$, then we define the linear constraint on the real part of the perturbation:
\begin{equation}
\label{constraint}
\langle \phi_0, p \rangle_{L^2_{\rm per}} = 0, \quad \phi_0 \equiv  \frac{\phi}{1-\phi^2}.
\end{equation}
The Morse index of $\mathcal{L}_+$ acting on $p$ changes under the constraint
and we study how it changes separately for the even and odd periodic waves. 

\subsection{Constrained energy minimization of even periodic solutions} 

Under the constraint (\ref{constraint}), we define the Morse and nullity indices 
of the constrained operator $\mathcal{L}_+ |_{\{ \phi_0\}^{\perp}}$ 
and denote them by $n(\mathcal{L}_+ |_{\{ \phi_0\}^{\perp}})$ and $z(\mathcal{L}_+ |_{\{ \phi_0\}^{\perp}})$. 

\begin{proposition}\label{propR0-constraint}
	$n(\mathcal{L}_+ |_{\{ \phi_0\}^{\perp}}) = 0$ and $z(\mathcal{L}_+ |_{\{ \phi_0\}^{\perp}}) = 1$ if and only if the mapping $\omega \to Q(\phi)$ is monotonically increasing at $\omega \in (\omega_L,1)$.
\end{proposition}

\begin{proof}
Since $\langle \phi_0, \phi' \rangle_{L^2_{\rm per}} = 0$, we have $\phi' \in \Ker(\mathcal{L}_+ |_{\{ \phi_0\}^{\perp}})$ by Proposition \ref{propR0}. It follows by \cite[Theorem 2.7]{GeyerPel25} that 
$$
n(\mathcal{L}_+ |_{\{ \phi_0\}^{\perp}}) = n(\mathcal{L}_+) - 1 = 0, \quad 
z(\mathcal{L}_+ |_{\{ \phi_0\}^{\perp}}) = z(\mathcal{L}_+) = 1
$$
if and only if 
$$
\langle \mathcal{L}_+^{-1} \phi_0, \phi_0 \rangle_{L^2_{\rm per}} < 0, 
$$
where equation (\ref{der-omega}) implies that 
$$
\langle \mathcal{L}_+^{-1} \phi_0, \phi_0 \rangle_{L^2_{\rm per}} = 
- \langle \phi_0, \frac{d \phi}{d \omega} \rangle_{L^2_{\rm per}} =- \frac{1}{2} \frac{d}{d\omega} Q(\phi) = - \frac{1}{2} \frac{d^2}{d \omega^2} G(\phi).
$$ 
This completes the proof of the assertion.
\end{proof}	

Propositions \ref{propL2} and \ref{propR0-constraint} imply the following result, which yields the assertion of Theorem \ref{th-stability} for even periodic waves.

\begin{corollary}
		The Hessian operator $\mathcal{L}$ defined by (\ref{matrixop}) in  $\mathbb{L}^2_{\rm per}$ with domain $\mathbb{H}^2_{\rm per}$ under the constraint (\ref{constraint}) is non-negative and admits a double zero eigenvalue with associated eigenfunctions $(\phi',0)$ and $(0,\phi)$ if and only if the mapping $\omega \to Q(\phi)$ is monotonically increasing at $\omega \in (\omega_L,1)$. 
\end{corollary}

\subsection{Constrained energy minimization of odd periodic solutions} 

We recall the definition (\ref{Y-space}) for $\mathcal{Y}\subset H^1_{\rm per}$ spanned by functions which are odd with respect to the half-period. 
We define the Morse and nullity indices of the constrained operator $\mathcal{L}_- |_{\mathcal{Y}}$ and denote them by $n(\mathcal{L}_- |_{\mathcal{Y}})$ and $z(\mathcal{L}_- |_{\mathcal{Y}})$.
Under the additional constraint (\ref{constraint}), we define the Morse and nullity indices of the constrained operator $\mathcal{L}_+ |_{\{ \phi_0 \}^{\perp}  \cap \mathcal{Y}}$ 
and denote them by $n(\mathcal{L}_+ |_{\{ \phi_0\}^{\perp}  \cap \mathcal{Y}})$ and $z(\mathcal{L}_+ |_{\{ \phi_0\}^{\perp}  \cap \mathcal{Y}})$. 

\begin{proposition}\label{propRR-constraint}
	$n(\mathcal{L}_+ |_{\{ \phi_0\}^{\perp} \cap \mathcal{Y}}) = z(\mathcal{L}_+ |_{\{ \phi_0\}^{\perp}  \cap \mathcal{Y}}) = 0$ if and only if the mapping $\omega \to Q(\phi)$ is monotonically increasing at $\omega \in (\Omega_L,1)$. Furthermore, $n(\mathcal{L}_- |_{\mathcal{Y}}) = 0$ and  $z(\mathcal{L}_- |_{\mathcal{Y}}) = 1$.
\end{proposition}

\begin{proof}
	Since $\phi' \notin \mathcal{Y}$ and $\phi \in \mathcal{Y}$, we have $\phi' \notin \Ker(\mathcal{L}_+ |_{\mathcal{Y}})$ and $\phi \in \Ker(\mathcal{L}_- |_{\mathcal{Y}})$ so that $z(\mathcal{L}_+ |_{\mathcal{Y}}) = 0$ and $z(\mathcal{L}_- |_{\mathcal{Y}}) = 1$. Since the eigenfunctions of $\mathcal{L}_+$ and $\mathcal{L}_-$ for the smallest (negative) eigenvalue are even with respect to the half-period, we also have 
	$n(\mathcal{L}_+ |_{\mathcal{Y}}) = 1$ and 	$n(\mathcal{L}_- |_{\mathcal{Y}}) = 0$. In addition, we have $\phi_0 \in \mathcal{Y}$. It follows by \cite[Theorem 2.7]{GeyerPel25} that 
	$$
	n(\mathcal{L}_+ |_{\{ \phi_0 \cap \mathcal{Y} \}^{\perp}}) = n(\mathcal{L}_+ |_{\mathcal{Y}}) - 1 = 0, \quad 
	z(\mathcal{L}_+ |_{\{ \phi_0 \cap \mathcal{Y} \}^{\perp}}) = z(\mathcal{L}_+ |_{\mathcal{Y}})  = 0
	$$
	if and only if 
	$$
	\langle \mathcal{L}_+^{-1} \phi_0, \phi_0 \rangle_{L^2_{\rm per}} < 0, 
	$$
	where equation (\ref{der-omega}) implies again that 
	$$
	\langle \mathcal{L}_+^{-1} \phi_0, \phi_0 \rangle_{L^2_{\rm per}} = 
	- \langle \phi_0, \frac{d \phi}{d \omega} \rangle_{L^2_{\rm per}} =- \frac{1}{2} \frac{d}{d\omega} Q(\phi) = - \frac{1}{2} \frac{d^2}{d \omega^2} G(\phi).
	$$ 
	This completes the proof of the assertion.
\end{proof}	

Proposition \ref{propRR-constraint} implies the following result, which yields the assertion of Theorem \ref{th-stability} for odd periodic waves.

\begin{corollary}
	
	The Hessian operator $\mathcal{L}$ defined by (\ref{matrixop}) in  $\mathbb{L}^2_{\rm per}$ with domain $\mathbb{H}^2_{\rm per} \cap \mathcal{Y}$ under the constraint (\ref{constraint}) is non-negative and admits a simple zero eigenvalue with the associated eigenfunction $(0,\phi)$ if and only if the mapping $\omega \to Q(\phi)$ is monotonically increasing at $\omega \in (\Omega_L,1)$. 
\end{corollary}

\section{Numerical approximations}
\label{sec-6}

Given a fixed $\omega \in (0,1)$, the energy level of homoclinic orbit $\mathcal E_\omega \in (0,\infty)$ is computed, and then  the period function $T(\mathcal{E},\omega)$ for the even and odd periodic waves is approximated separately by using (\ref{even-period}) and (\ref{odd-period}), respectively. 
The period function is plotted on Figure \ref{fig-period}.

For the even waves, since the period function diverges as $\mathcal E \to \mathcal E_\omega^- $, the grid on $(0,\mathcal E_\omega)$ are defined in two regions $(\mathcal E_\omega-10^{-3},\mathcal E_\omega)$ with $2000$ equally spaced grid points and $(0, \mathcal E_\omega-10^{-3})$ with $300$ 
equally spaced grid points. For the odd waves, the grids are defined analogously as on $(\mathcal E_\omega, \mathcal E_\omega +10^{-2})$ with $100$ grid points and $( \mathcal E_\omega +10^{-2}, 0.5)$ with $300$ grid points. We evaluate the integrals with the absolute and relative tolerances given by  $\epsilon_\mathrm{abs}=10^{-10}$ and $\epsilon_\mathrm{rel}=10^{-8}$ respectively. Selected values $\omega =0.3,0.5,0.7,0.9$ are plotted in Figure \ref{fig-period} with $T = 2\pi \sqrt{\frac{1-\omega}{2\omega}}$ at $\mathcal{E}= 0$ represented by solid dots.

Once the period function $T(\mathcal{E},\omega)$ is computed, we fix the spatial 
period $L > 0$ and find the uniquely defined energy level $\mathcal{E}_L(\omega)$ from a root of $T(\mathcal{E}_L(\omega), \omega) = L$. This is possible due to monotonicity of the period function with respect to $\mathcal{E}$ in Theorem \ref{th-period}. We use Newton's root-finding method for a grid $\{\omega_j\}_{j=1}^{M}$ of values of $\omega$ in either $(\omega_L,1)$ or $(\Omega_L,1)$, see Theorems \ref{th-existence} and \ref{th-period}. We thus obtain the values $\{ \mathcal{E}_j \}_{j=1}^M$ for $\mathcal{E}_L(\omega_j)$, which are plotted on the left panels of Figures \ref{fig-even} and \ref{fig-odd} relative to $\mathcal{E}_{\omega}$, for $\tilde{\mathcal{E}}_L(\omega) = \mathcal{E}_L(\omega) - \mathcal{E}_{\omega}$. Thus, the solid dots for $\omega = \omega_L$ correspond to $\tilde{\mathcal{E}}_L(\omega) = - \mathcal{E}_{\omega}$ in Figure \ref{fig-even} 
and the solid dots for $\omega = \Omega_L$ correspond to $\tilde{\mathcal{E}}_L(\omega) = 0$ on Figure \ref{fig-odd}.

Numerical inaccuracies occur in the computations of $\mathcal{E}_L(\omega)$ near $\omega = 1$ due to the wave profiles becoming steep, and this is independent of the grid $\{\omega_j\}_{j=1}^{M}$. The solid dots on the left panels of Figures \ref{fig-even} and \ref{fig-odd} show the end points for which the accuracy is verified within $10^{-8}$ computational error. The limiting values of  $\mathcal{E}_L(\omega)$ at $\omega = 1$ obtained from  (\ref{E-limeven}) and (\ref{E-limodd}) are shown by open dots on the left panels 
of  Figures \ref{fig-even} and \ref{fig-odd}. An interpolation is performed between the last numerical data for $\mathcal{E}_L(\omega)$ and the value of 
$\mathcal{E}_L(\omega = 1)$ and it is shown by the dotted line on the left panels  of Figures \ref{fig-even} and \ref{fig-odd}. 

For the computed set $\{ (\mathcal{E}_i,\omega_i)\}_{i = 1}^M$, the profile $\phi = \phi(x)$ of the even periodic wave satisfying (\ref{even-wave}) with $x_0 = 0$ 
is obtained by numerical integration of  
\begin{equation}
\label{F-even}
x =  F_{\rm even}(\phi) = \int^{M}_{\phi} \frac{d \phi}{\sqrt{2\mathcal E_L(\omega) - (\omega - \phi^2) - (1-\omega)\log\frac{1-\omega}{1-\phi^2}}}, \quad \phi \in [m,M],
\end{equation}
	where $m$ and $M$ are obtained from two positive roots of $V(\phi) = \mathcal{E}_L(\omega)$ for $\omega \in (\omega_L,1)$ and $\mathcal{E}_L(\omega) \in (0,\mathcal{E}_{\omega})$, see Remark \ref{boundphi}. The solution profile is defined implicitly as $x = F_{\rm even}(\phi) \in \left[0,\frac{L}{2}\right]$ with $\phi(0) = M$ and $\phi\left(\frac{L}{2}\right) = m$. It is extended from  $\left[0,\frac{L}{2}\right]$ to $\left[-\frac{L}{2},0\right]$ 
	by using the even reflection: $\phi(-x) = \phi(x)$. This yields the wave profiles on the right panel of Figure \ref{fig-even}. The dashed line shows the peaked profile at $\omega = 1$ given analytically by (\ref{phi-limeven}). 
	
	For the computed set $\{ (\mathcal{E}_i,\omega_i)\}_{i = 1}^M$, the profile $\phi = \phi(x)$ of the odd periodic wave satisfying (\ref{odd-wave}) with $x_0 = 0$ is obtained by numerical integration of  
\begin{equation}
\label{F-odd}
	x =  F_{\rm odd}(\phi ) = -\int_{0}^{\phi} \frac{d \phi}{\sqrt{2\mathcal E_L(\omega) - (\omega - \phi^2) - (1-\omega)\log\frac{1-\omega}{1-\phi^2}}}, \quad \phi \in [0,M],
\end{equation}
		where $M$ is obtained from the only positive root of $V(\phi) = \mathcal{E}_L(\omega)$ for $\omega \in (\Omega_L,1)$ and $\mathcal{E}_L(\omega) \in (\mathcal{E}_{\omega},\infty)$, see Remark \ref{remodd}. The solution profile is defined implicitly as $x = F_{\rm odd}(\phi) \in \left[0,\frac{L}{4}\right]$ with $\phi(0) = 0$ and $\phi\left(\frac{L}{4}\right) = M$. It is extended from $\left[0,\frac{L}{4}\right]$ to $\left[-\frac{L}{2},0\right]$ by using the symmetries of the odd periodic wave: $\phi(-x) = -\phi(x) = -\phi\left(\frac{L}{2} - x\right)$. This yields the wave profiles on the right panel of Figure \ref{fig-odd}. The dashed line shows the peaked profile at $\omega = 1$ given analytically by (\ref{phi-limodd}). 

We compute the mass $Q(\phi)$ shown in Figure \ref{fig-mass} versus $\omega$ by using the integration in the $\phi$ variable. For the even periodic wave, we use 
\begin{equation}
\label{Q-even}
Q (\phi) = - 2\int_{0}^{L/2} \log (1-\phi^2)\,dx = 2 \int^{M}_{\phi} \frac{\log(1-\phi^2)}{\sqrt{2\mathcal E_L(\omega) - (\omega - \phi^2) - (1-\omega)\log\frac{1-\omega}{1-\phi^2}}} d \phi.
\end{equation}
Computing the integral numerically for $\{ (\mathcal{E}_i,\omega_i)\}_{i = 1}^M$ 
yields the left panel of Figure \ref{fig-mass}. The numerical data are again 
missing near $\omega = 1$ and the last available data is shown by the solid dots, for which the accuracy of $10^{-8}$ is guaranteed. The open dots show the limiting 
values of $Q(\phi)$ at $\omega = 1$, which can be computed analytically as
\begin{align}
\label{Q-even-limit}
\omega = 1 : \quad Q(\phi) = 2L \log \left [2\cosh \left (\frac{L}{2}\right ) \right ] - L^2 + \frac{\pi^2}{6}  -  \mathrm{Li_2}(e^{-2L}), 
\end{align}
where $\mathrm{Li_2}$ denotes the dilogarithm function 
$$
\mathrm{Li_2}(z) := -\int_0^z \frac{\ln(1-u)}{u}\,du.
$$  
Interpolation between the last available data (right solid dots) and the limiting value of $Q(\phi)$ at $\omega = 1$ (open dots) is shown by the dotted line on Figure \ref{fig-mass}.

The dashed line on the left panel of Figure \ref{fig-mass} shows the limiting value of $Q(\phi)$ versus $\omega$ in the soliton case with $L = \infty$, for which the integral for $Q(\phi)$ is still computed on the compact interval. The dependence of $Q(\phi)$ versus $\omega$ is similar to the periodic case $L < \infty$ and displays a single maximum before the peak for which 
$$
\omega = 1, \quad L = \infty : \quad  Q(\phi) =  \frac{\pi^2}{6}.
$$

For the odd periodic wave, we use 
\begin{equation}
\label{Q-odd}
Q (\phi) = - 4\int_{0}^{L/4} \log (1-\phi^2)\,dx = 4 \int_0^{\phi} \frac{\log(1-\phi^2)}{\sqrt{2\mathcal E_L(\omega) - (\omega - \phi^2) - (1-\omega)\log\frac{1-\omega}{1-\phi^2}}} d \phi.
\end{equation}
Computing the integral numerically for $\{ (\mathcal{E}_i,\omega_i)\}_{i = 1}^M$ 
yields the right panel of Figure \ref{fig-mass}. The limiting 
value of $Q(\phi)$ at $\omega = 1$ is computed analytically as 
\begin{align}
\label{Q-odd-limit}
\omega = 1 : \quad Q(\phi) = 2L \log \left [2\sinh \left (\frac{L}{4}\right ) \right ] - \frac{L^2}{2} + \frac{\pi^2}{3}  -  2\mathrm{Li_2}(e^{-L}).
\end{align}
We note that 
$$
\omega = 1, \quad L = \infty : \quad  Q(\phi) =  \frac{\pi^2}{3}
$$
which is double compared to the case of the even periodic wave. This corresponds to the fact that the odd periodic wave represents two solitons on a single period for large $L$. 

Table \ref{table-Q} represents the numerical values of $Q(\phi)$  used in Figure \ref{fig-mass} for $\omega = 1$. These numerical values are computed from (\ref{Q-even-limit}) and (\ref{Q-odd-limit}). 
\begin{table}[h!]
\centering
\begin{tabular}{|c|c|c|}
\hline
$L$ & $Q$ (even, $\omega = 1$) & $Q$ (odd, $\omega = 1$) \\
\hline
$2\pi$ & 1.66837567259328 & 2.73100651970082 \\
\hline
$3\pi$ & 1.64645514903036 & 3.11961052401896 \\
\hline
$4\pi$ & 1.64502171315626 & 3.24288332619890 \\
\hline
\end{tabular}
\vspace{0.25cm}
\caption{The numerical values of $Q(\phi)$ used in Figure \ref{fig-mass} for $\omega = 1$.}
\label{table-Q}
\end{table}

Finally, we expand Remark \ref{rem-data} to discuss the three-branched behavior of $Q(\phi)$ versus $\omega$ in the soliton limit $L = \infty$ observed 
in \cite{KPR24} and disputed in Figure \ref{fig-mass}. We cannot reproduce 
the three-branched behavior by using (\ref{Q-even}) and (\ref{Q-odd}). Even if we take fewer number of grid points, we would evaluate $Q(\phi)$ with a lower accuracy but still observe the two-branched behavior of $Q(\phi)$ versus $\omega$ in Figure \ref{fig-mass}.

The reason for the three-branched behavior of $Q(\phi)$ observed in \cite{KPR24} is due to the finite-difference approximation applied to the differential equation (\ref{ode}) and to the Hessian operator $\mathcal{L}$ in (\ref{matrixop}) with the uniform grid of $x$ values. The larger grid spacing leads to inaccurate computations of $\phi$ near the maximum $\phi(0) = M$ and results in highly inaccurate computations of $Q(\phi)$. 

We fix $\omega \in (0,1)$ and consider the differential equation (\ref{ode}) 
on the truncated interval $[-L,L]$ with $L = 20$. Since the bright solitons 
decay exponentially to zero at infinity, we can use the Dirichlet boundary conditions $\phi(\pm L) = 0$. We replace $[-L,L]$ by the uniform grid of $N$ 
points $\{x_i\}_{i=1}^{N}$ with the spacing $\Delta x = \frac{2L}{N-1}$ 
and compute approximations for the solution profile  $\{\phi_i\}_{i=1}^N$ with $\phi_1 = \phi_N= 0$. The second derivative can be constructed using the central difference method as $\{ (D^2\phi)_j \}_{j=2}^{N-1}$ given by 
$$
(D^2\phi)_j= \frac{\phi_{j-1} - 2\phi_j + \phi_{j+1}}{(\Delta x)^2}, \quad \quad j = 2,\dots N-1.
$$
The residual of the differential equation (\ref{ode}) is defined by
$$
R_j = (1-\phi_j^2) (D^2 \phi)_j - (\omega - \phi_j^2) \phi_j, \quad \quad j = 2,\dots N-1,
$$
and we introduce the mapping $T: \mathbb R^{N-2} \to \mathbb R^{N-2}$ such that $T(\phi) = R$. The first derivative of the mapping is given by the Jacobian matrix $J = \nabla T \in \mathbb R^{N-2 \times N-2}$ with the nonzero elements given by 
$$
J_{j,j\pm1} = \frac{1-\phi_j^2}{(\Delta x)^2}, \quad \quad J_{j,j} = -\frac{2(1-\phi_j^2)}{(\Delta x)^2} - 2\phi_j (D^2\phi)_j - \omega + 3\phi_j^2, \quad 2 \leq j \leq N-2. 
$$

To minimize the residual $\Phi(\phi) = \frac{1}{2} \| T(\phi) \|^2$, we implement the linear Newton's method in the iterations $\{ \phi^{(k)}\}_{k = 0}^{\infty}$ defined by 
$ J (\phi^{(k+1)} - \phi^{(k)}) = -T(\phi^{(k)})$ starting with a suitable initial guess 
$$
\phi^{(0)}_j = \min \{0.9,\sqrt{2\omega} \} \mathrm{sech} (\sqrt{\omega}x_j), \quad  j = 2,\dots N-1. 
$$
To avoid overshoot, we perform backtrack line search by starting from $a = 1$ and reducing to find $a \in (0,1]$ that satisfies the decreasing condition 
$$
\Phi (\phi^{(k)} + a\psi^{(k)}) \leq \Phi (\phi^{(k)})(1-ca), \quad \mbox{\rm where} \;\; \psi^{(k)} = - J^{-1} T(\phi^{(k)}),
$$
for a small $c = 10^{-4}$. When this is achieved, we accept and update 
the next iteration as $\phi^{(k+1)} =\phi^{(k)} + a \psi^{(k)}$, after which we compute $J^{(k+1)}$, $T(\phi^{(k+1)})$, and $\psi^{(k+1)}$. The algorithm is terminated when the convergence condition $\Vert T(\phi^{(k+1)}) \Vert /\sqrt{N-2} \leq \epsilon_\mathrm{tol}$ with a small tolerance $\epsilon_\mathrm{tol} = 10^{-8}$. This iterative method yields the solution profile $\{(x_j,\phi_j)\}_{j=1}^{N}$, from which we compute the mass integral $Q(\phi)$  by using the trapezoidal method.

Figure \ref{fig-discrete} shows the plot of $Q(\phi)$ versus $\omega$ for two spacings $\Delta x = 0.1$ and $\Delta x = 0.2$, compared to the dependence computed from (\ref{Q-even}) in the limit $L \to \infty$ (dashed line). 
The latter dependence is interpreted as the limit $\Delta x \to 0$ 
in the finite-difference method. The finite-difference approximation with $\Delta x > 0$ for the differential equation (\ref{ode}) leads to the three-branched behavior reported in \cite{KPR24}. We computed the mass integral for the values of $\omega$ in $[0.005, 0.93]$ on an equally spaced grid of $100$ points. Since the numerical data are not accurate near $\omega = 1$, we perform the quadratic extrapolation to extend the values of the mass integral from the last numerical data at $\omega = 0.93$ into the interval $[0.93,1]$.

Thus, we conclude that the three-branched behavior of $Q(\phi)$ versus $\omega$ is a numerical artefact of the finite-difference method.

\begin{figure}[!ht]
	\centering
	\includegraphics[width=0.7\textwidth,height=0.4\textheight]{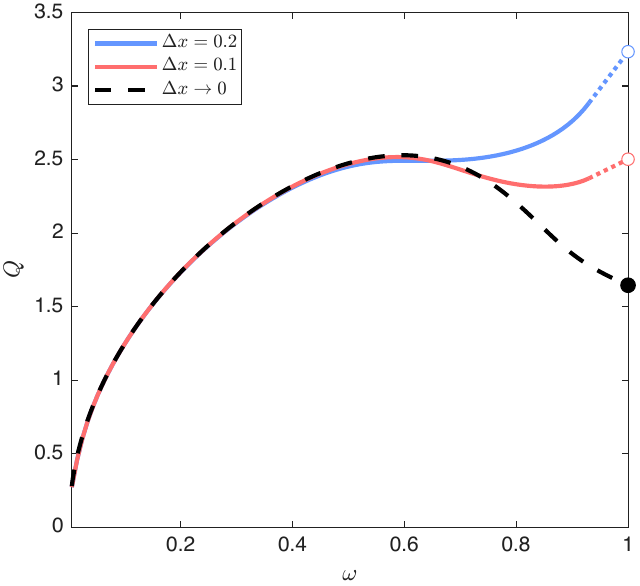} 
	\caption{The dependence of the mass integral $Q(\phi)$ 
		computed by the finite-difference method versus $\omega$ for $\Delta x = 0.1, 0.2$. The dashed line shows the same dependence computed by using (\ref{Q-even}) for $L \to \infty$. } 
	\label{fig-discrete}
\end{figure}

{\bf Acknowledgement.} We would like to thank P. G. Kevrekidis for suggesting that the three-branched behavior of the mass versus the frequency observed in \cite{KPR24} can be explained as a numerical artefact of the finite-difference method and for preparing a preliminary version of Figure \ref{fig-discrete}. A part of this work is performed as the undergraduate BSc thesis of S. Wang at McMaster University.

   \end{document}